\documentclass{article}
\usepackage{setspace}
\usepackage{color}
\usepackage{amsfonts}
\usepackage{amsmath}
\usepackage{amsthm}
\usepackage{amssymb}

\newcommand{\bn}{\mathbb N}

\newcommand{\br}{\mathbb R}

\newcommand{\nc}{\newcommand}
\newtheorem{lemma}{Lemma}[section]
\newtheorem{theorem}{Theorem}[section]

\newtheorem{definition}{Definition}[section]
\newtheorem{corollary}{Corollary}[section]

\newtheorem{remark}{Remark}[section]

\nc{\cD}{{\cal D}}
\nc{\cP}{{\cal P}}
\nc{\cR}{{\cal R}}
\nc{\e}{\varepsilon}
\nc{\Om}{\Omega}
\nc{\om}{\omega}
\nc{\aal}{\alpha}
\nc{\tow}{\rightharpoonup}
\nc{\nin}{\in \hs{-.35}/\,}
\nc{\np}{\newpage}
\nc{\g}{\gamma}
\nc{\IN}{I \hs{-.15} N}
\nc{\IR}{I \hs{-.14} R}
\nc{\IK}{I \hs{-.14} K}
\nc{\hs}[1]{\hspace{#1cm}}
\def\C#1{{\mathcal {#1}}}
\title{\bf Global attractors for multivalued semiflows with weak continuity properties}

\author{Piotr Kalita
    \thanks{E-mail : piotr.kalita@ii.uj.edu.pl}
    \thanks{The research was supported by the Marie Curie International Research Staff
Exchange Scheme Fellowship within the 7th European Community Framework
Programme under Grant Agreement No. 295118, by the National Science
Center of Poland under grant no. N N201 604640, and by the International
Project co-financed by the Ministry of Science and Higher Education of
Republic of Poland under grant no. W111/7.PR/2012.},
Grzegorz {\L}ukaszewicz
    \thanks{E-mail : glukasz@mimuw.edu.pl, Tel.: +48 22 55 44 562}
\thanks{This research was supported by Polish Government Grant N N201 547638} }

\date{ }

\begin{document}
\maketitle
\begin{center}

 {\small
$^{\star\dag}$Faculty of Mathematics and Computer Science,
Institute of Computer Science,
Jagiellonian University,
ul. prof. S. {\L}ojasiewicza 6, 30-348 Krak\'ow, Poland,

$^{\ddag\S}$University of Warsaw, Mathematics Department,
     ul.Banacha 2, 02-957 Warsaw, Poland}

\end{center}
\normalsize

\begin{abstract}
\noindent A method is proposed to prove the global attractor existence for multivalued semiflows with weak continuity properties.
An application to the reaction-diffusion problems with nonmonotone multivalued semilinear boundary condition and nonmonotone multivalued semilinear source term is presented.
\end{abstract}

\vspace{0.2cm}


\vspace{0.2cm} \noindent{\it 1991 Mathematics Subject
Classification:} 76D05, 76F10, 76F20, 47J20, 49J40 \vspace{0.2cm}
 \renewcommand{\theequation}{\arabic{section}.\arabic{equation}}
 \setcounter{equation}{0}

  \section{Introduction}
There are at least three approaches to prove the existence of global attractor for the problems without uniqueness of solutions:
 method of multivalued semigroups or multivalued semiflows developed in the ground-breaking paper od Babin and Vishik \cite{babin1985} (see also \cite{Melnik-1998}),
 method of generalized semiflows (see \cite{Ball-1997}), and method of trajectory attractors (see \cite{Chepyzhov1995}, \cite{Chepyzhov1997}, \cite{malek-necas-1996}, \cite{Sell}).
 Method of trajectory attractors that was related to the other two ones in \cite{Kapustyan2010} relies on the study of shift operators on the sets of time dependent trajectories while the other two approaches, which are discussed in relation to each other in \cite{Caraballo-2003}, consist in the direct study of the sets of states obtainable from the given initial conditions after some period of time. This mapping, known as multivalued semigroup or multivalued semiflow ($m$-semiflow) is denoted as $\br^+\times H \ni (t,x)\to G(t,x)\subset 2^H$, where $H$ is the suitable
 Banach (or metric) space of the problem states. In order to show the existence of a compact global attractor, i.e. the compact set in $H$ that is invariant (or sometimes only negatively semiinvariant) and attracts all bounded sets in $H$ three properties are required: existence of a set that is bounded in $H$ and absorbs all trajectories of $G$ after some finite time, some compactness type property of $G$ and some continuity or closedness type property of $m$-semiflow $x\to G(t,x)$.
 Clasically, this last property is the upper semicontinuity (in the sense of multifunctions) whith respect to the strong topology in the argument space and strong topology in the value space (see \cite{Ball-1997}, \cite{Melnik-1998}). 
 In the work of Zhong, Yang and Sun \cite{Zhong-Yang-Sun} the approach to show existence of a global attractor for problems
 governed by semiflows (i.e. problems with the uniqueness of solutions) which are only strong-weak continuous is presented. This approach was futher
 developed in \cite{Wang-Zhou-2007} where some results for nonautonomous strict $m$-semiflows are shown.

The present paper is on one hand the extension of the results of \cite{Ball-1997} and \cite{Melnik-1998} since the condition of
 semiflow upper semicontinuity is relaxed to the condition called ($NW$) in the sequel and on the other hand the extension of works
 \cite{Zhong-Yang-Sun} and \cite{Wang-Zhou-2007} to a more general, multivalued case. The motivation for the introduction of this condition is twofold: firstly, as it is shown in Lemmata \ref{lemma:nw_1} and \ref{lemma:ex3_NW_H} below, it is natural to verify for the problems with multifunctions having the form of Clarke subdifferential since it follows from basic a priori estimates and passing to the limit argument; secondly it can replace the strong - strong upper semicontinuity and graph closedness in the abstract theorems on the attractor existence even if  an $m$-semiflow is only point dissipative and nonstrict (see Corollary \ref{corollary:main} and Theorem \ref{thm:pointd} below).

 Note that in \cite{Wang-Zhou-2007} the extension of the approach of \cite{Zhong-Yang-Sun} to the case of $m$-semiflows
 that are strict (i.e. such that $G(t+s,x) = G(t,G(s,x))$ for all $x\in H$ and $s,t\in \mathbb{R}^+$) is proposed,
 while in the present study we consider the case where only the inclusion $G(t+s,x)\subset G(t,G(s,x))$ is assumed to hold.
 Moreover we propose another generalization of strong-weak continuity than \cite{Wang-Zhou-2007}, namely for $x_n\to x$ strongly in $H$
 and $\xi_n\in G(t,x_n)$ we assume that $\xi_n$ must have a weakly convergent subsequence while in \cite{Wang-Zhou-2007} it is assumed that the whole sequence $\xi_n$ must converge.

Attractors for partial differential equations and inclusions without uniqueness were studied in the recent articles of Kasyanov \cite{Kasyanov2011}, \cite{Kasyanov2012} where the approaches by $m$-semiflows and trajectory attractors were used for first order autonomous evolution equations and inclusions with general nonlinear pseudomonotone operators. The results were adapted to second order evolution inclusions and hemivariational inequalities in \cite{Kasyanov2012b}, \cite{Zgurovsky2012a}. Note, that in \cite{Kasyanov2011}, \cite{Kasyanov2012}, \cite{Kasyanov2012b}, \cite{Zgurovsky2012a} the strong-strong uppersemicontinuity of $m$-semiflow is always used and the compactness is proved by the analysis of the energy function monotonicity.

Another interesting recent article on existence of global attractors for $m$-semiflows is the paper of Coti Zelati \cite{coti-zelati},
 where only the strict case is considered and the semiflow closedness
 is assumed to hold only at some time instant $t^*>0$ and not for all $t\geq 0$.

Examples presented in the present study show that the condition ($NW$), that states that the multivalued semiflow has weakly compact values and is strong-weak upper semicontinuous, is natural to check for the problems governed by differential inclusions where the multivalued term has the form of Clarke subdifferential.

 For an exhaustive review of recent results on the theory of asymptotic behavior for
 problems without uniqueness of solutions see \cite{Balibrea-2012}. The difficulty
 in the analysis of these problems lies in the fact that it remains unknown if
 every solution can be obtained as the limit of the solutions of approximative problems (for example Galerkin problems)
 and, in consequence, the estimates that hold for the approximate solutions do not have to hold for all
 solutions of the original problem (see Section 4.3.1 in \cite{Balibrea-2012}). It must be remarked here that
 while most authors consider only the existence of attractors for the multivalued semiflows, there are almost no results
 on the attractor properties, like their dimension, the attraction speed or the attractor structure. The notable exceptions are the article of Arrieta et. al. \cite{Arrieta-2006} where for one dimensional nonlinear reaction-diffusion problem it is shown that attractor consists of heteroclinic connections between a countable number of fixed points, the article of Kapustyan et. al. \cite{Kapustyan2012} where the characterization of an attractor for the problem governed by the nonlinear reaction-diffusion equation by means of stable and unstable manifolds of the rest points is given and the article of Kasyanov et al. \cite{Kasyanov2013} where the regularity of all weak solutions and their attractors for reaction-diffusion type evolution inclusions were studied.
 
The plan of this article is the following: in Section \ref{sec:theory} we present abstract results on the attractor existence while in Section \ref{sec:applications} we present examples of the problems for which we show the attractor existence by means of proposed abstract framework.

 \section{Abstract theory of global attractors for multivalued semiflows with a weak continuity property}\label{sec:theory}

 Let $H$ be a Banach space, and $P(H)$ be the family of all nonempty subsets of $H$. Some definitions and results of this section remain valid for more general setup of metric spaces and in such cases it will be explicitly noted that $H$ is only a metric space. By $B(x,r)$ we will denote the closed ball centered in $x\in H$ with the radius $r\in\mathbb{R}^+$. Note that here, and in the sequel of this paper we denote $\br^+=[0,\infty)$.

If $H$ is a metric space equipped with the metric $\rho(\cdot,\cdot)$, then for $x\in H$ and $B\subset H$, we set $\mbox{dist}_H(x,B)=\inf_{y\in B}\rho(x,y)$. Moreover if $A,B\subset H$ then we define the Haussdorff semidistance from $A$ to $B$ by $\mbox{dist}_H(A,B)=\sup_{x\in A}\mbox{dist}_H(x,B)$. Same definitions are valid for normed spaces with $\rho(x,y)$ replaced by $\|x-y\|$.

\begin{definition}
The map $G: \br^+\times H\to P(H)$ is called a multivalued semiflow ($m$-semiflow) if:

(1) $G(0,z)=z$ for all $z\in H$.

(2) $G(t+s,z) \subset G(t,G(s,z))$ for all $z\in H$ and all $t,s\geq 0$.

\end{definition}

\begin{definition}
An $m$-semiflow is strict if $G(t+s,z) = G(t,G(s,z))$ for all $z\in H$ and all $t,s\geq 0$.
\end{definition}
\subsection{Measure of noncompactness and its properties}
We recall the definition and some properties of the Kuratowski measure of noncompactness, cf. \cite{deimling-1985}.

\begin{definition}
Let $H$ be a complete metric space and $A$ be a bounded subset of $H$.
The Kuratowski measure of noncompactness $\kappa(A)$ of $A$ is defined by
$$\kappa(A) = \inf\{\delta > 0: A\,\, \mathrm{has\,\, a\,\, finite\,\, open \,\,cover\,\, of \,\,sets\,\, of \,\,diameter\,\,} < \delta\}.$$
If $A$ is a nonempty, unbounded set in $H$, then we define $\kappa(A)=\infty$.
\end{definition}

\begin{lemma} \label{Kuratowski-prop}
The Kuratowski measure of noncompactness $\kappa(A)$ on a complete metric
space $H$ satisfies the following properties:
\begin{itemize}
\item[(1)] $\kappa(A) = 0$ if and only if $\bar{A}$ is compact, where $\bar{A}$ is the closure of $A$;

\item[(2)] If $A_1 \subset A_2$, then $\kappa(A_1) \leq \kappa(A_2)$;

\item[(3)] $\kappa(A_1 \cup A_2) \leq \max\{\kappa(A_1), \kappa(A_2)\}$;

\item[(4)] $\kappa(\bar{A}) = \kappa(A)$;

\item[(5)] If $A_t$ is a family of nonempty, closed, bounded sets defined for $t>r$, where $r\in \mathbb{R}^+$, that satisfy
    $A_t \subset A_s$, whenever $s\leq t$, and $\kappa(A_t) \to 0$, as $t\to\infty$, then
    $\bigcap_{t>r}A_t$ is a nonempty, compact set in $H$.
\end{itemize}
\noindent If, in addition, $H$ is a Banach space, then the following are valid:
\begin{itemize}
\item[(6)] $\kappa(A_1 + A_2)\leq\kappa(A_1) + \kappa(A_2)$;

\item[(7)] $\kappa(co(A)) = \kappa(A)$, where $co(A)$ is the closed convex hull of $A$;

\item[(8)] Let $H$ have the decomposition
    $H = H_1 + H_2$, with $\mathrm{dim}\ H_1 < \infty$,
    $P : H \to H_1$, $Q : H \to H_2$ be the canonical projectors, and $A\subset H$ be bounded.
     If the diameter of $Q(A)$ is less then $\epsilon$, then $\kappa(A) < \epsilon$.
     \end{itemize}
\end{lemma}

\begin{definition} \label{def-ws-cl}
 Let $A$ be a subset of Banach space $H$. The weak sequential closure
  $\bar{A}^{ws}$ of $A$ is defined by
$$\bar{A}^{ws} = \{x \in H: \,\,\mbox{there exists the sequence}\,\, \{x_n\} \subset A,\,\, \mbox{such that}\,\, x_n
 \to x \,\,\mbox{weakly in}\,\, H\}.$$
\end{definition}
Note that in general topological space, $\bar{A}^{ws}$ is different from $\bar{A}$ or the weak closure
$\bar{A}^w$ of $A$.
But if $A$ is a convex subset of a Banach space, then we know that $\bar{A} = \bar{A}^w = \bar{A}^{ws}$.
\begin{lemma} \label{lemme1.2} (see Lemma 2.4 in \cite{Zhong-Yang-Sun})  Let $H$ be a Banach space and $\kappa$ be the Kuratowski measure of noncompactness.
Then for any subset $A$ of $H$, we have
  $\kappa(A) = \kappa(\bar{A}^{ws})$.
\end{lemma}

\subsection{Compactness of multivalued semiflows}
We define three compactness type properties of multivalued semiflows and
investigate in which spaces these properties coincide.

\begin{definition}
Let $H$ be a complete metric space. The multivalued semiflow $G:\mathbb{R}^+\times H\to P(H)$ is $\omega$-limit compact if for every bounded set $B\subset H$ we have
$$
     \kappa\left(\bigcup_{t\geq\tau} G(t,B)\right) \to 0, \quad \tau\to\infty.
$$
\end{definition}

\begin{definition}
Let $H$ be a complete metric space. The multivalued semiflow $G:\mathbb{R}^+\times H\to P(H)$ is asymptotically compact if for every bounded set $B\subset H$, and for all sequences $t_n\to \infty$ and
$\xi_n\in G(t_n,B)$, there exists a subsequence $\{\xi_{n_k}\}$ such that $\xi_{n_k}\to\xi$ strongly in $H$ for some $\xi\in H$.
\end{definition}

\begin{definition}
Let $H$ be a Banach space. The multivalued semiflow $G:\mathbb{R}^+\times H\to P(H)$ satisfies the flattening condition if for every bounded set $B\subset H$ and $\epsilon>0$ there exists
$t_0(B,\epsilon)$ and a finite dimensional subspace $E$ of $H$ such that for a bounded projector $P:H\to E$,
the set $P\left(\bigcup_{t\geq t_{0}} G(t,B)\right)$ is bounded in $H$ and
$$(I-P)\left(\bigcup_{t\geq t_{0}} G(t,B)\right) \subset B(0,\epsilon).$$
\end{definition}

The notion of $\omega$-limit compactness is used to study attractors for single valued semiflows for
example in \cite{Zhong-Yang-Sun}, and was generalized to strict multivaled semiflows in \cite{Wang-Zhou-2007}. The notion of asymptotic compactness for $m$-semiflows is used in \cite{Melnik-1998}. Note, however, that in \cite{Melnik-1998} the term ''asymptotic compactness'' is not used directly, it is assumed that $m$-semiflow should be asymptotically upper semicompact and eventually bounded but the conjunction of these two notions is equivalent to asymptotic compactness (see \cite{Caraballo-2003} for the discussion of the relation between these notions). Next two lemmata show that in complete metric spaces, for $m$-semiflows which are not necessarily strict, $\omega$-limit compactness is equivalent to asymptotic compactness.

\begin{lemma}\label{lemma:limit-asymptotic}
If the $m$-semiflow $G$ on the complete metric space $H$ is $\omega$-limit compact then it is asymptotically compact.
\end{lemma}
\begin{proof}
Let $B$ be a bounded set in $H$ and let $\tau_n$ be such that
\begin{eqnarray}
    \kappa\left(\bigcup_{t\geq\tau_n}G(t,B)\right) \leq \frac{1}{n}, \quad \tau_n\to\infty.
\end{eqnarray}
Let $t_i\to\infty$, $\xi_i\in G(t_i,B)$.
We shall prove that $\kappa(\{\xi_i\}_{i=1}^\infty)=0$. For every $k\in\bn$ we have
$$\kappa(\{\xi_i\}_{i=1}^\infty)=\kappa(\{\xi_i\}_{i=1}^k \cup \{\xi_i\}_{i=k+1}^\infty)
\leq\kappa(\{\xi_i\}_{i=k+1}^\infty).$$
For every $n\in \bn$ and $k$ such that $t_{k+1}\geq\tau_n$ we have $\kappa(\{\xi_i\}_{i=k+1}^\infty)\leq \frac{1}{n}$, whence $\kappa(\{\xi_i\}_{i=1}^\infty)=0$.  This proves, in turn, the precompactness of $\{\xi_i\}_{i=1}^\infty$, and, in consequence, the asymptotic compactness of $G$.
\end{proof}

The proof of next lemma uses the idea from the proof of Theorem 1 in \cite{Melnik-1998}.

\begin{lemma}\label{lemma:asymptoimpliesomega}
If the $m$-semiflow $G$ on the complete metric space $H$ is asymptotically compact then it is $\omega$-limit compact.
\end{lemma}
\begin{proof}
Let $G$ be asymptotically compact on $H$. We shall prove first that for every bounded set $B$ in $H$ the set
\begin{equation}\label{eq:omega}\omega(B) = \bigcap_{t\geq 0}\overline{\bigcup_{s\geq t}G(s, B)}\end{equation}
is nonempty.

Indeed, let $t_n\to\infty$, $\xi_n\in G(t_n,B)$ and let, for a subsequence, still denoted by $n$, $\xi_n \to \xi$.
For every $\tau\geq 0$ and every index $n$ such that $t_n\geq \tau$, we have
$\xi_n\in \bigcup_{t\geq \tau}G(t, B)$. Since $\xi_n\to\xi$, then for every $\tau \geq 0$,
$\xi\in \overline{\bigcup_{t\geq \tau}G(t, B)}$. Thus $\xi\in \omega(B)$.

Now we prove that for every bounded set $B\subset H$ we have $\mathrm{dist}_H(G(t,B), \omega(B)) \to 0$ as $t\to\infty$. Assume, to the contrary, that there exists the bounded set $B_0\in H$ and the sequences $t_n\to\infty$ and $\xi_n\in G(t_n,B_0)$ such that $\mathrm{dist}_H(\xi_n, \omega(B_0)) \geq\epsilon>0$.
By the asymptotic compactness property, for a subsequence, still denoted by $n$, we have
$\xi_n\to \xi$ in $H$. But $\xi\in \omega(B)$, which gives a contradiction.

 Let $B\subset H$ be bounded and let $x_n$ be a sequence in $\omega(B)$. We prove that this sequence has a subsequence that converges to some element in $\omega(B)$ and thus the set $\omega(B)$ is compact. As
$$x_n\in \overline{\bigcup_{t\geq s}G(t, B)} \quad \mathrm{\,\,for\,\,every\,\,} s\geq 0,$$
then, for any sequence $t_n\to \infty$ there exists $\xi_{k_n}\in G(t_{k_n},B)$ such that
$\rho(x_n,\xi_{k_n})\leq\frac{1}{n}$. But, by asymptotic compactness, there exists a subsequence $\xi_\nu$ of $\xi_{k_n}$ converging to some $\xi\in \omega(B)$. Thus, also $x_\nu\to \xi$.

Now let us fix $\epsilon >0$. We need to show that there exists $t_0>0$ such that the set $\bigcup_{t\geq t_0}G(t,B)$ can be covered by
 finite number of sets with diameter $\epsilon$. From compactness of $\omega(B)$ it follows that there exists finite number of points $\{x_i\}_{i=1}^n$ such that
$
\omega(B) \subset \bigcup_{i=1}^n B(x_i,\frac{\epsilon}{2}).
$
Now choose $t_0>0$ such that $\mathrm{dist}_H(G(t,B), \omega(B))<\frac{\varepsilon}{2}$ whenever $t\geq t_0$. It follows that for all $t\geq t_0$
we have $G(t,B) \subset \bigcup_{i=1}^n B(x_i,\epsilon)$ and the proof is complete.
\end{proof}

Note that in the course of the proof of Lemma \ref{lemma:asymptoimpliesomega} we have shown the following corollary

\begin{corollary} \label{corollary-1}
Let $H$ be a complete metric space space and let the multivalued semiflow $G$ on $H$ be
$\omega$-limit compact. Then for every bounded set $B$ in $H$ its $\omega$-limit set $\omega(B)$ defined by
(\ref{eq:omega}) is a nonempty and compact set such that $\mathrm{dist}_H(G(t,B),\omega(B))\to 0$ as $t\to\infty$.
\end{corollary}

Next two lemmata relate the flattening condition with $\omega$-limit compactness.
Note that they generalize Theorem 3.10 in \cite{Ma-Wang-Zhong} to the case of $m$-semiflows.
The analogous results for the multivalued, strict and nonautonomous case are established in Theorem 4.10 in \cite{Wang-Zhou-2007}.

\begin{lemma} \label{thm-flat-limit}
If the $m$-semiflow $G$ on the Banach space $H$ satisfies the flattening condition then $G$ is $\omega$-limit compact.
\end{lemma}

\begin{proof} The proof follows the lines of the proof of assertion (1) in Theorem 3.10 in
\cite{Ma-Wang-Zhong}. Let us choose $B\subset H$ bounded and $\epsilon>0$. Using the flattening condition and assertions (6) and (8) of Lemma \ref{Kuratowski-prop} we can find $t_0(B,\epsilon)>0$ such that
\begin{eqnarray*}
     &\kappa\left(\bigcup_{t\geq t_0}G(t,B)\right) \leq \kappa\left(P\left(\bigcup_{t\geq t_0}G(t,B)\right)\right) +
     \kappa\left((I-P)\left(\bigcup_{t\geq t_0}G(t,B)\right)\right) \leq\\
     &\leq  \kappa(B(0,\epsilon)) = 2\epsilon.
\end{eqnarray*}
Whence, $G$ is $\omega$-limit compact.
\end{proof}

\begin{lemma}
Let $G$ be a multivalued semiflow on a uniformly convex Banach space $H$. If $G$ is $\omega$-limit compact then it satisfies the flattening condition.
\end{lemma}
\begin{proof}
The proof does not touch any continuity properties of $G$ and we follow strictly the lines of the proof of assertion (2) in Theorem 3.10 in
\cite{Ma-Wang-Zhong}. Let $B$ be a bounded set in $H$. By the $\omega$-limit compactness property, for every $\epsilon>0$ there exists $t(B,\epsilon)>0$ such that
$$
    \bigcup_{t\geq t(B,\epsilon)}G(t,B) \subset \bigcup_{i=1}^n A_i
$$
for some sets $A_i$ of diameter less then $\epsilon$. Let $x_i\in A_i$. Then
$$
    \bigcup_{t\geq t(B,\epsilon)}G(t,B) \subset \bigcup_{i=1}^n B(x_i, \epsilon).
$$
Denote $H_1=\mathrm{span}\{x_1,...,x_n\}$. Since $H$ is uniformly convex, there exists a projection
$P:H\to H_1$ such that for every $x\in H$, $\|x-Px\|=\mathrm{dist}(x,H_1)$. Hence,
    $$(I-P)\bigcup_{t\geq t(B,\epsilon)}G(t,B) \subset B(0,\epsilon),$$ which proves the flattening condition.
\end{proof}
Summarizing the results of this subsection, we have
\begin{align*}
\mbox{in complete metric spaces}&\ \ \ (\mbox{asymptotic compactness})\Longleftrightarrow(\omega\mbox{-limit compactness})\\
\mbox{in Banach spaces}&\ \ \ (\mbox{flattening condition})\Longrightarrow(\omega\mbox{-limit compactness})\\
\mbox{in uniformly convex Banach spaces}&\ \ \ (\mbox{flattening condition})\Longleftrightarrow(\omega\mbox{-limit compactness})
\end{align*}

\subsection{Continuity and closedness of multivalued semiflows}

\begin{definition} Let $H$ be a normed space and $X$ be a topological space. The multifunction $G:H\to 2^X$ is upper semicontinuous
if for every sequence $x_n\to x$ in $H$ and for every open set $V\subset X$ such that $G(x)\subset V$ there exists $n_0\in \mathbb{N}$ such that
$G(x_n)\subset V$ for all $n\geq n_0$.
\end{definition}

\begin{definition}
The multivalued semiflow $G:\mathbb{R}^+\times H\to P(H)$ on the Banach space $H$ is \textit{closed}
if for all $t\geq 0$ the graph of the multivalued mapping $x\to G(t,x)$ is closed in strong-strong topology.
\end{definition}

\begin{definition}
The multivalued semiflow $G:\mathbb{R}^+\times H\to P(H)$ on the Banach space $H$ is \textit{demiclosed}
if for all $t\geq 0$ the graph of the multivalued mapping $x\to G(t,x)$ is closed in strong-weak topology.
\end{definition}

Note that, obviously, every demiclosed $m$-semiflow is closed.

We introduce the condition ($NW$), ''norm-to-weak'', that generalizes to the multivalued case the norm-to-weak continuity assumed in \cite{Zhong-Yang-Sun} for semigroups (see Definition 3.4 in \cite{Zhong-Yang-Sun}).

\begin{definition}
The multivalued semiflow $G:\mathbb{R}^+\times H\to P(H)$ on the Banach space $H$ satisfies the condition ($NW$) if for every $t\geq 0$ from $x_n\to x$ in $H$ and
$\xi_n\in G(t,x_n)$ it follows that there exists a subsequence $\{\xi_{k_n}\}$, such that $\xi_{k_n}\to \xi$ weakly in $H$ with $\xi\in G(t,x)$.
\end{definition}

Note that the similar condition is assumed for the nonautonomous multivalued case in \cite{Wang-Zhou-2007} (see condition (3) in Definition 2.6 in \cite{Wang-Zhou-2007}), where the strict semiflow is considered and, instead of a subsequence, whole sequence is assumed to converge weakly. Next lemma provides the characterization of the condition ($NW$) in Banach spaces.

\begin{lemma}
Let $H$ be a Banach space. The multivalued semiflow $G:\mathbb{R}^+\times H\to P(H)$ satisfies the condition ($NW$) if and only if for all $(t,x)\in \mathbb{R}^+\times H$ the set $G(t,x)$ is weakly compact and for all $t\in \mathbb{R}^+$ the multifunction $G(t,\cdot)$ is strong-weak upper semicontinuous.
\end{lemma}
\begin{proof}
The proof follows the lines of the proof of Proposition 4.1.11 in \cite{denkowski2003}. First assume that $G$ satisfies the condition ($NW$).
Let us choose $(t,x)\in \mathbb{R}^+\times H$ and the sequence $\xi_n\in G(t,x)$. For a subsequence, $\xi_\nu\to \xi$ weakly in $H$ with
$\xi \in G(t,x)$ and hence $G(t,x)$ is weakly compact. We pass to the proof of upper semicontinuity. Let $x_n\to x$ strongly in $H$
and $V\subset H$ be a weakly open set such that $G(t,x)\subset V$. We continue the proof by contradiction. Assume that there exists the subsequence
$x_\nu$ and the sequence $\xi_\nu \in G(t,x_\nu)$ such that $\xi_\nu\not\in V$ for all indices $\nu$. From the condition ($NW$) we can choose another subsequence, still denoted by $\nu$ such that $\xi_\nu\to\xi$ weakly in $H$ and $\xi\in G(t,x)$ and moreover $\xi\in V$. However, since $H\setminus V$ is weakly closed and $\xi_\nu \in H\setminus V$ then $\xi\in H\setminus V$ and we have the contradiction.

Now we assume that $G(t,\cdot)$ is strong-weak upper semicontinuous and weakly compact valued. We need to show that the condition ($NW$) holds.
We take $x_n\to x$ strongly in $H$ and $\xi_n\in G(t,x_n)$. We continue by contradiction. Assume that for every $\eta\in G(t,x)$ we can find
the index $n_0$ and weak neighbourhood $V(\eta)$ such that $\xi_n\not\in V(\eta)$ for all $n\geq n_0$. The family $\{V(\eta)\}_{\eta\in G(t,x)}$
is a weakly open cover of a weakly compact set $G(t,x)$. Hence we have
$$
G(t,x)\subset \bigcup_{i=1}^n V(\eta_i)\equiv V.
$$
We are able to find the index $N_0$ such that for all $n\geq N_0$ we have $\xi_n\not\in V$. Since $V$ is weakly open, from upper semicontinuity it follows that there exists $m_0\in\mathbb{N}$ such that $G(t,x_n)\subset V$ for all $n\geq m_0$. Hence for $n\geq \max\{N_0,m_0\}$ we have
$$
V\not\ni \xi_n\in G(t,x_n)\subset V,
$$
a contradiction.
\end{proof}

\begin{lemma} \label{lemma-closed}
If $G:\mathbb{R}^+\times H\to P(H)$ on a Banach space $H$ is a multivalued semiflow satisfying condition ($NW$) then it is also demiclosed.
\end{lemma}
\begin{proof}
An elementary proof follows directly from the definitions.
\end{proof}

Summarizing the results of this subsection, we have, for an $m$-semiflow $G$ on a Banach space
\begin{align*}
(G(t,x)\,\, \mbox{is weakly compact and}&\,\, G(t,\cdot)\,\,\mbox{is strong-weak upper semicontinous})\Leftrightarrow (NW)\\
&\Downarrow\\
(G\,\,&\mbox{is demiclosed}) \Longrightarrow (G\,\,\mbox{is closed})
\end{align*}
\subsection{Abstract results on global attractor existence}

\begin{definition}
The set $A\subset H$ is called a global attractor for $G$ if:

(1) $A$ is compact in $H$.

(2) $A\subset G(t,A)$ for all $t\geq 0$ ($A$ is negatively semi-invariant).

(3) For every bounded $B\subset H$, $A$ attracts $B$ ($\mathrm{dist}_H(G(t,B),A) \to 0, \quad t\to\infty$).
\end{definition}

Now we formulate the theorem on attractor existence. The necessity is a known result, it is proved in Lemma 3 in \cite{Melnik-2008} (see also \cite{Melnik-1998}, \cite{Zgurovsky2012}). We provide the proof here since it is different than in \cite{Melnik-2008}: it relies on the weak topology arguments rather than strong topology ones. The argumentation relying on weak topology is related to condition ($NW$) and the method we apply in the examples. Moreover we provide also the sufficient condition.

Note that in the single valued case this theorem gives equivalent conditions for global attractor existence for the continuous semigroup (see Theorem 3.9 in \cite{Ma-Wang-Zhong}).

\begin{theorem} \label{thm-main1}
Let $H$ be a Banach space and let the multivalued semiflow $G:\mathbb{R}^+\times H\to P(H)$ be closed. Then there exists a global attractor for $G$ if and only if
\begin{itemize}
\item[(i)] $G$ has a bounded absorbing set $B_0$ in $H$ (i.e. there exists a bounded set $B_0\in H$ such that for any bounded set $B\subset H$ there exists $t_0>0$ such that $\bigcup_{t\geq t_0}G(t,B)\subset B_0$);
\item[(ii)] $G$ is $\omega$-limit compact.
\end{itemize}
\end{theorem}

\begin{proof} The argument below generalizes that in \cite{Zhong-Yang-Sun} applied to norm-to-weak continuous semigroups. The proof that the existence of the global attractor implies (i) and (ii) is the same as that provided in \cite{Zhong-Yang-Sun}.

We shall prove the sufficiency of (i) and (ii) for existence a global attractor for closed multivalued semigroups.

\noindent \textbf{Step 1.} We define the candidate set for a global attractor by
\begin{eqnarray}
     A = \bigcap_{t\geq 0}\overline{\bigcup_{s\geq t}G(s, B_0)}^{ws}
\end{eqnarray}
where $B_0$ is a bounded absorbing set for $G$. We first prove that
\begin{eqnarray} \label{attractor-sequence}
     \xi\in A \Longleftrightarrow \mathrm{\,\,there\,\,exist\,\,} t_n\to\infty\,\mbox{and}\, \xi_n\in G(t_n,B_0)\,\mbox{such that}\,
     \xi_n\to\xi \mathrm{\ weakly\,\, in \,\,} H.
\end{eqnarray}
$(\Leftarrow)$ For every $\tau\geq 0$ and every $t_n\geq \tau$ we have $\xi_n\in G(t_n,B_0)$, and moreover
$\xi_n\in \bigcup_{t\geq \tau}G(t, B_0)$. From the weak convergence $\xi_n\to\xi$ it follows that for every $\tau \geq 0$,
$\xi\in \overline{\bigcup_{t\geq \tau}G(t, B_0)}^{ws}$. Thus $\xi\in A$.

\noindent
$(\Rightarrow)$ For $\xi\in A$ we have $\xi\in \overline{\bigcup_{t\geq n}G(t, B_0)}^{ws}$ for every $n\in \mathbb{N}$. From the definition of weak sequential closedness it follows that for every $n\in \mathbb{N}$ there exist sequences $\{t_n^k\}_{k\in\mathbb{N}}$ such that $t_n^k\geq n$ and
$\xi_n^k\in G(t_n^k,B_0)$ with $\xi_n^k\to \xi$ weakly in $H$ as $k\to\infty$. Consider the set
$K=\{\xi_n^k: k,n=1,2,3,...\}$. Since $G$ is $\omega$-limit compact, the set $\bar{K}^{ws}$ is weakly compact in $H$. Observe that $U=\overline{span\, K}$ is a separable and closed subspace of $H$ and $\bar{K}^{ws}$ is weakly compact in $U$. Hence, from Theorem 3.6.24 in \cite{denkowski2003} it follows that $\bar{K}^{ws}$ is metrizable with a metric $d(\cdot,\cdot)$ generating the weak topology in $\bar{K}^{ws}$. Therefore, for every $n\in\mathbb{N}$ there exists $\xi_n^{k_n}$ such that $d(\xi_n^{k_n},\xi)<\frac{1}{n}$. Hence we have $\xi_n^{k_n} \to \xi$ weakly in $H$ as $n\to\infty$ with $\xi_n^{k_n}\in G(t_n^{k_n},B_0)$ and $t_n^{k_n}\to\infty$ as $n\to\infty$, which ends the proof.

\noindent\textbf{Step 2.} $A$ is nonempty and compact. We have, by $\omega$-limit compactness of $G$ and by
Lemma \ref{lemme1.2},
\begin{eqnarray}
    \kappa\left(\overline{\bigcup_{t\geq \tau}G(t, B_0)}^{ws}\right) =
    \kappa\left(\bigcup_{t\geq \tau}G(t, B_0)\right) \to 0, \,\, \tau\to\infty.
\end{eqnarray}
Since the sets $\overline{\bigcup_{t\geq \tau}G(t, B_0)}^{ws}$ are nonempty, bounded and closed in $H$, we can apply Lemma \ref{Kuratowski-prop} (5) to get the claims.

\noindent\textbf{Step 3.} Attraction property. Since every bounded set in $H$ is absorbed by $B_0$ after some time, it suffices to prove that $\mathrm{dist}_H(G(t,B_0), A) \to 0$, $t\to\infty$. Assume, to the contrary, that there exist $t_n\to\infty$ and $\xi_n\in G(t_n,B_0)$ such that $\mathrm{dist}_H(\xi_n, A) \geq\epsilon>0$.
By the $\omega$-limit compactness property, there exists a convergent subsequence,
$\xi_\mu\to \xi$ strongly in $H$. As moreover $\xi_\mu\to\xi$ weakly in $H$, from (\ref{attractor-sequence}) we get $\xi\in A$, which gives a contradiction.

\noindent\textbf{Step 4.} We prove that $A\subset G(t,A)$ for all $t\geq 0$. Let $x\in A$ and $t>0$. We shall prove that
$x\in G(t,p)$ for some $p\in A$. In view of (\ref{attractor-sequence}), there exist $t_n\to\infty$ and $\xi_n\in G(t_n,B_0)$ such that $\xi_n\to x$ weakly in $H$. By the $\omega$-limit compactness property, there exists a subsequence such that $\xi_\mu\to \xi$ strongly in $H$ and $\xi=x$.
 We have,
$$\xi_\mu\in G(t_\mu,B_0) = G(t+(t_\mu-t),B_0) \subset G(t,G(t_\mu-t, B_0)),$$
whence there exist $z_\mu\in B_0$ such that $\xi_\mu \in G(t, G(t_\mu-t,z_\mu))$, and
$p_\mu\in G(t_\mu-t,z_\mu)$ such that $\xi_\mu \in G(t, p_\mu)$. By the $\omega$-limit compactness it follows that for a subsequence of $\{p_\mu\}$ denoted by the same symbol we have $p_\mu \to p$ in $H$. From (\ref{attractor-sequence}) it follows that and that $p\in A$.
Since $\xi_\mu \to \xi$ and $p_\mu\to p$ strongly in $H$ with $\xi_\mu\in G(t,p_\mu)$ from the fact that $G$
is closed it follows that $\xi\in G(t,p)$. Thus $x\in G(t,A)$, whence $A\subset G(t,A)$.
\end{proof}

\begin{remark}
Under the assumptions of Theorem \ref{thm-main1} if we define
\begin{eqnarray}
    \tilde{A} = \bigcap_{t\geq 0}\overline{\bigcup_{s\geq t}G(s, B_0)},
\end{eqnarray}
then $\tilde{A}=A$. Indeed, from Definition \ref{def-ws-cl}, for any set $B$ we have, $\bar{B}\subset \bar{B}^{ws}$. Hence
$\tilde{A}\subset A$. Assume that there exists $\xi\in A\setminus\tilde{A}$. By (\ref{attractor-sequence}) there exists a sequence $t_n\to\infty$ and $\xi_n\in G(t_n,B_0)$ such that $\xi_n\to \xi$ weakly in $H$. From $\omega$-limit compactness it follows that there exists a subsequence $\xi_\mu$ of the sequence $\xi_n$ such that $\xi_\mu \to \xi$ strongly in $H$, and this means that $\xi\in \tilde{A}$, which gives a contradiction. Thus, $A=\tilde{A}$. Moreover, the global attractor is the minimal closed set attracting every bounded set in $H$, cf. \cite{Melnik-1998}.
\end{remark}

   In Section \ref{sec:applications} we will present two examples for which it is most convenient to obtain the global attractor existence for the problems governed by subdifferential inclusions through the following corollary of Theorem \ref{thm-main1}.

\begin{corollary}\label{corollary:main}
If the multivalued semiflow $G$ on a uniformly convex Banach space $H$ has a bounded absorbing set, satisfies the flattening condition and condition ($NW$) then there exists a global attractor for $G$.
\end{corollary}

Now we present some additional consequences of Theorem \ref{thm-main1}. The following argument shows, that Theorem 4 from \cite{Melnik-1998} can be obtained as a corollary of Theorem~\ref{thm-main1}.
\begin{theorem}
If $G$ is a closed multivalued semiflow on a Banach space $H$, and there exists a compact set $K$ attracting all bounded sets then $G$ has a global attractor. It is a minimal closed set attracting bounded sets in $H$.
\end{theorem}
\begin{proof}
 As $K$ attracts all bounded sets $B$ in $H$, its $\epsilon$-neighbourhood ${\cal O}_\epsilon(K)$ is a bounded absorbing set. Let $G(t,B)\subset {\cal O}_\epsilon(K)$ for $t\geq t(B)$. Then
\begin{eqnarray}
      \kappa\left(\bigcup_{t\geq t(B)}G(t,B)\right) \leq \kappa({\cal O}_\epsilon(K)) \leq 2\epsilon.
\end{eqnarray}
Thus, $G$ is closed, $G$ has a bounded absorbing set and is $\omega$-limit compact. From
Theorem~\ref{thm-main1} it follows that $G$ has a global attractor. The minimality property is then well known, cf. \cite{Melnik-1998}.
\end{proof}

\begin{definition}
The $m$-semiflow $G$ is point dissipative if there exists a bounded set $B_0\subset H$ such that for every $x\in H$ we have $\mathrm{dist}_H(G(t,x),B_0)\to 0$ as $t\to\infty$, i.e. $B_0$ attracts every point of $H$.
\end{definition}
The necessity in the following theorem valid for strict closed $m$-semiflows follows from Lemma 2 in \cite{Melnik-2008}.
\begin{theorem} \label{thm-main3}
Let $G$ be a strict closed multivalued semiflow on a Banach space $H$.
Then there exists a global attractor for $G$ if and only if
\begin{itemize}
\item[(i)] $G$ is point dissipative.
\item[(ii)] $G$ is $\omega$-limit compact.
\end{itemize}
\noindent The attractor is minimal among all closed sets attracting all bounded sets of $H$.
\end{theorem}
\begin{proof}
From Theorem \ref{thm-main1} it follows that if $G$ has a global attractor then $G$ is $\omega$-limit compact and there
exists a bounded absorbing set. This last property implies that $G$ is point dissipative.
The existence of global attractor follows from Lemma 2 in \cite{Melnik-2008} (see also Lemma 1.4 in \cite{Zgurovsky2012}).
\end{proof}
As a simple consequence of Theorems \ref{thm-main1} and \ref{thm-main3} we formulate the following corollary
\begin{corollary}
Let $G$ be a strict closed multivalued semiflow on a Banach space $H$, that is moreover $\omega$-limit compact. Then $G$ is point dissipative if and only if it has a bounded absorbing set.
\end{corollary}
Results of Melnik and Valero (\cite{Melnik-1998}, \cite{Melnik-2008}, \cite{Zgurovsky2012}) show that if we assume that
$m$-semiflow is point dissipative without assuming the existence of bounded absorbing set, then the compact attractor exists either if the semiflow is strict (see above Theorem \ref{thm-main3}) or if it is upper semicontinuous with respect to the strong-strong topology (see Theorem 1 in \cite{Melnik-2008}). The following theorem shows that another situation where we do not have to assume the existence of a bounded absorbing set is when an $m$-semiflow satisfies 
the condition ($NW$) or (in reflexive case) it is demiclosed.
\begin{theorem}\label{thm:pointd}
Let $G$ be a multivalued semiflow on a Banach space $H$ such that the following conditions hold
\begin{itemize}
\item[(i)] $G$ satisfies the flattening condition;
\item[(ii)] $G$ is point dissipative;
\item[(iii)] at least one of the following two conditions holds:
\begin{itemize}
\item[(A)] $G$ satisfies the condition ($NW$);
\item[(B)] $H$ is reflexive and $G$ is demiclosed.
\end{itemize}
\end{itemize}
Then $G$ has a bounded absorbing set, and, in consequence, $G$ has a compact global attractor.
\end{theorem}

\begin{proof}
Let $B_0$ be a bounded set that attracts every point in $H$. Fix $\varepsilon >0$. We define
$B_1=\C{O}_\varepsilon(B_0)$ and $B_2=\C{O}_\varepsilon(\omega(B_1)).$ Obviously $B_2$ is bounded. We will show that $B_2$ is absorbing.
Assume, for the sake of contradiction, that there exists a bounded set $B$ and sequences $t_n\to\infty$ and $x_n\in G(t_n,B)$ such that $x_n\not\in B_2$ for all $n\in\mathbb{N}$. Hence, we can find a sequence $\{y_n\}\subset B$ such that $x_n\in G(t_n,y_n)$. Moreover $x_n\in G\left(\frac{t_n}{2},G\left(\frac{t_n}{2},y_n\right)\right)$ and hence there exists a sequence $\{\xi_n\}$ such that
\begin{equation}\label{eq:weak1}
\xi_n\in G\left(\frac{t_n}{2},y_n\right)
\end{equation}
 and \begin{equation}\label{eq:weak2}
 x_n \in G\left(\frac{t_n}{2},\xi_n\right).
 \end{equation}
 Since by Lemmata \ref{lemma:limit-asymptotic} and \ref{thm-flat-limit} $G$ is asymptotically compact, from (\ref{eq:weak1}) it follows that,
  for a subsequence still denoted by $n$, we have $\xi_n\to \xi$ strongly in $H$. We denote $S=\{\xi_n\}_{n=1}^\infty\cup\{\xi\}$. Obviously the set $S$
  is bounded in $H$, hence, from a flattening condition there exists $t_1>0$, $R>0$ and a finite dimensional subspace $E\subset H$
  as well as bounded projector $P_E:H\to E$  such that we have
   $P_E\left( \bigcup_{t\geq t_1}G(t,S)\right) \subset B(0,R)$ and \begin{equation}\label{eq:weak3}(I-P_E)\bigcup_{t\geq t_1}G(t,S)\subset B\left(0,\frac{\varepsilon}{4}\right).\end{equation} From point dissipativity of $G$ it follows that
   we can choose $t_2>0$ such that \begin{equation}\label{eq:weak4}\bigcup_{t\geq t_2}G(t,\xi)\subset \C{O}_{\frac{\varepsilon}{4}}(B_0).\end{equation}
   Now let $T=\max\{t_1,t_2\}$. From (\ref{eq:weak2}) it follows that, for a sufficiently large $n$ we have $x_n\in G\left(\frac{t_n}{2}-T,G(T,\xi_n)\right)$ and hence there exists a sequence $z_n\in G(T,\xi_n)$ such that \begin{equation}\label{eq:weak5}x_n\in G\left(\frac{t_n}{2}-T,z_n\right).\end{equation} We can decompose $z_n=P_Ez_n + (I-P_E)z_n$. Since $\{P_Ez_n\}$ is a bounded sequence in the finite dimensional space $E$, then, for a subsequence, $P_Ez_n\to z^1$ strongly in $H$.
Now we proceed separately for cases $(iii)(A)$ and $(iii)(B)$.

If $(iii)(A)$ holds, then, from the condition ($NW$) it follows that, for a subsequence
we have $z_n\to z$ weakly in $H$ with $z\in G(T,\xi)$. Moreover $(I-P_E)z_n\to z-z^1$ weakly in $H$ and, from (\ref{eq:weak3}), by weak lower semicontinuity of the norm it follows that $\|z-z^1\|\leq \frac{\varepsilon}{4}$.

If in turn $(iii)(B)$ holds, then from (\ref{eq:weak3}) it follows that $\|(I-P_E)z_n\|\leq \frac{\varepsilon}{4}$, and by reflexivity of $H$, for a subsequence we have $(I-P_E)z_n\to z^2$ weakly in $H$. From weak lower semicontinuity of the norm it follows that $\|z^2\|\leq \frac{\varepsilon}{4}$. Moreover, we have $z_n\to z$ weakly in $H$, where $z=z^1+z^2$. From demiclosedness of $G$ it follows that $z\in G(T,\xi)$.

We continue the proof for both cases. We have, for sufficiently large $n$

$$
\|z_n-z\| \leq \|(I-P_E)z_n\| + \|P_Ez_n-z^1\|+\|z-z^1\|\leq \frac{3\varepsilon}{4}.
$$
From (\ref{eq:weak4}) it follows that, for sufficiently large $n$ we have $z_n\in B_1$. Hence, by (\ref{eq:weak5}) and asymptotic compactness of $G$ we get that, for a subsequence, $x_n\to x$ strongly where $x\in \omega(B_1)$. Therefore, for sufficiently large $n$ we have $x_n\in B_2$, a contradiction.
\end{proof}

As a simple consequence of above theorem we formulate the following Corollary
\begin{corollary}
Let $G:\mathbb{R}^+\times H\to P(H)$ be a demiclosed multivalued semiflow on a uniformly convex Banach space $H$.
The semiflow $G$ has a global attractor if and only if the following two conditions hold
\begin{itemize}
\item[(i)] $G$ satisfies the flattening condition,
\item[(ii)] $G$ is pointwise dissipative.
\end{itemize}
\end{corollary}

\renewcommand{\theequation}{\arabic{section}.\arabic{equation}}
\setcounter{equation}{0}
\section{Applications}\label{sec:applications}
 The examples presented in this section contain multivalued terms in the form of a Clarke subdifferentials. The notion of a Clarke subdifferential can be
 defined for a locally Lipschitz functional $j:H\to \mathbb{R}$, where $H$ is a Banach space \cite{Clarke}. The Clarke subdifferential $\partial j$ at the point $x\in H$ is defined as  $$\partial j(x)=\{\xi\in H^* \ |\  \langle \xi,x\rangle_{H^*\times H}\leq j^0(x;v)\},$$ where $j^0(x;v)$ is a generalized Clarke directional derivative
 given by $$j^0(x;v)=\limsup_{\lambda\to 0^+, y\to x}\frac{j(y+\lambda v)-j(y)}{\lambda}.$$ If $H$ is finite dimensional, then $\partial j$ has a relatively simple characterization, namely
 $$
 \partial j(x)=\overline{conv} \{\lim_{i\to\infty}\nabla j(x_i)\ |\ x_i\to x, x_i\notin S\},
 $$
 where $S$ is the set of measure zero on which $j$ fails to be differentiable \cite{Clarke}. Differential inclusions in which the multivalued terms have a form of
 Clarke subdifferentials are known as hemivariational inequalities and they are used to model the multivalued semipermeability laws \cite{Miettinen1999} or multivalued  friction and normal compliance laws in mechanics \cite{Migorski2006}. For an overview of the theory of hemivariational inequalities see the monoghraphs \cite{Naniewicz1995}, \cite{carl}, \cite{Migorski2013}.

  We also recall in this place the Translated Gronwall Lemma proved in \cite{lukaszewicz-2010}.
  \begin{lemma}\label{lemma:gronwall_lukaszewicz}
  If for some $\lambda>0$ and $h>0$ and functions $h:\br^+\to \br^+$ and $y:\br^+\to \br^+$ we have
  $$
  y'(s)+\lambda y(s)\leq h(s)\,\,\mbox{for a.e.}\,\,s\geq 0,
  $$
  where $y, y', h$ are locally integrable functions, then
  $$
  y(t+2)\leq e^{-\lambda}\int_t^{t+1}y(s)\, ds + e^{-\lambda(t+2)}\int_t^{t+2}e^{\lambda s}h(s)\, ds,
  $$
  for all $t\geq 0$. If, in particular, $h$ is a constant function, then
  $$
  y(t+2)\leq e^{-\lambda}\int_t^{t+1}y(s)\, ds + \frac{h}{\lambda},
  $$
  for all $t\geq 0$.
  \end{lemma}

We would like to note that the problems considered in this section satisfy the 'strong continuity properties' (the m-semiflows are strong-strong uppersemicontinuous in $L^2(\Omega)$), and thus could be treated by the method used in \cite{Kasyanov2011}, \cite{Kasyanov2012}.
However, our alternative approach by condition ($NW$) has this advantage that it does not require additional arguments than the standard a priori estimates and passing to the limit (see for example Lemma 3.7 below). Moreover, our approach allows to show the attractor existence in $L^2(\Omega)$ and $L^p(\Omega)$ for the problem in which the elliptic operator is coercive with the exponent $2$ and the multivalued term satisfies the growth condition with the exponent $p$ (see Section 3.2 below) which is a certain generalization of the corresponding results from the mentioned papers. 
The approach to compactness by flattening condition considered here is an alternative to the one by monotonicity of energy function.

 \subsection{Attractor in $L^2(\Omega)$ for heat equation with multivalued boundary conditions}

\noindent The example of this subsection is motivated by the temperature control problem considered for example in \cite{wang-yang-2010},
where the heat flow through some section of domain boundary is associated with the temperature by a multivalued feedback control law.

\noindent Note that the existence of global attractor for the problem considered in this section follows (after the technical generalization from $f\in H$ to $f\in V^*$) from the results of \cite{Kasyanov2011} and \cite{Kasyanov2012} where the approach by strong-strong upper semicontinuity and monotonicity of auxiliary energy function was used for the problems governed by more general class of pseudomonotone operators.

 \noindent  Let $\Omega\subset \br^n$ be an open and bounded set with sufficiently smooth boundary. The boundary $\partial \Omega$ is divided into two parts
 $\Gamma_D$ and $\Gamma_M$, where $m(\Gamma_D)>0$. Let $V=\{u\in H^1(\Omega)\ |\  u=0\ \mbox{on}\ \Gamma_D\}$ and $H=L^2(\Omega)$. We define an operator
 $A:V\to V^*$ as $\langle Au,v\rangle = \int_\Omega \nabla u(x)\cdot \nabla v(x)\ dx$, where $\langle\cdot,\cdot\rangle$ is duality in $V\times V^*$.  The norm in
 $V$ is denoted as $\|v\|^2=\langle Av,v\rangle$. Norms in other spaces then $V$ will be denoted by appropriate subscripts. The linear and continuous trace mapping leading from $V$ to $L^2(\Gamma_M)$ is denoted by $\gamma$, we will use the same symbol to denote the Nemytskii trace operator on the spaces of time dependent functions. The norm of trace operator will be denoted by $\|\gamma\|\equiv\|\gamma\|_{\C{L}(V;L^2(\Gamma_M))}$.
 If $u\in L^2(\Gamma_M)$ we will denote by $S^2_{\partial j(u)}$ the set of all $L^2$ selections of $\partial j(u)$ i.e. all functions $\xi\in L^2(\Gamma_M)$ such that $\xi(x) \in \partial j(u(x))$ for a.e. $x\in \Gamma_M$.

 The problem under consideration is the following

 \noindent \textbf{Problem ($\mathcal{P}_1$).} Find $v\in L^2_{loc}(\br^+;V)$ with $v'\in L^2_{loc}(\br^+;V^\star)$, $v(0)=v_0$
 such that
 \begin{eqnarray} \label{def_sol_ex_2}
    &&  \langle v'(t) + Av(t),z \rangle + (\xi(t),\gamma z)_{L^2(\Gamma_M)}
      = \langle F, z \rangle ,\label{eq:example_11}\\
    &&  \xi(t)\in S^2_{\partial j(\gamma v(t))}, \nonumber
 \end{eqnarray}
 for a.e. $t\in\br^+$ and for all $z\in V$.

 The above problem is in fact a weak form of the following initial and boundary value problem
  \begin{eqnarray} \label{def_sol_classical_2}
    &&  v'(x,t)-\Delta v(x,t) = F(x)\ \mbox{in}\ \Omega\times \br^+,\\
    &&  v(x,0) = v_0(x)\ \mbox{in}\ \Omega,\\
    &&  v(x,t) = 0 \ \mbox{on}\ \Gamma_D\times \br^+,\\
    &&  -\frac{\partial v(x,t)}{\partial \nu} \in \partial j(v(x,t)) \ \mbox{on}\ \Gamma_M\times \br^+.
 \end{eqnarray}

 The assumptions on problem data are the following
 \begin{itemize}
 \item[$H_0$:] $v_0\in H, F\in V^*$,
 \item[$H(j)$:] $j:\br\to \br$ is a locally Lipschitz function such that
 \begin{itemize}
 \item[$(i)$] $\partial j$ satisfies the growth condition $|\xi| \leq a+b|s|$ for all $s\in \br$ and $\xi\in \partial j(s)$ with $a,b>0$,
 \item[$(ii)$] $\partial j$ satisfies the dissipativity condition $\xi s \geq c - d |s|^2$ for all $s\in \br$ and $\xi\in \partial j(s)$ with $c\in \br$ and $d\in \left(0,\frac{1}{\|\gamma\|^2}\right)$,
 \end{itemize}
 \end{itemize}

 Note that assuming $H(j)(i)$ for $u\in L^2(\Gamma_M)$ the set $S^2_{\partial j(u)}$
 is in fact the set of all measurable selections of $\partial j(u)$.

Existence of solutions to Problem ($\mathcal{P}_1$) follows from the results of \cite{Migorski2004}. Note that the assumption $H(j)(ii)$
is more general than the corresponding assumption of \cite{Migorski2004} (namely $H(j)(iv)$), however, the careful examination of
the proof in \cite{Migorski2004} reveals that $H(j)(ii)$ is sufficient for the existence of solutions. See also \cite{Kalita2012submitted}
for the existence analysis under the hypothesis $H(j)(ii)$ in the general setup.

\begin{theorem}
Problem ($\mathcal{P}_1$) has a solution.
\end{theorem}

We associate with Problem ($\mathcal{P}_1$) the multifunction $G:\br^+\times H\to P(H)$, which assigns to the time $t$ and initial condition $v_0$ the set of states
attainable from $v_0$ after time $t$. Observe that $G$ constitutes an $m$-semiflow.

\begin{lemma}\label{lemma:basic_a_priori}
If $v\in L^2_{loc}(\br^+;V)$ with $v'\in L^2_{loc}(\br^+;V^\star)$ solves Problem ($\mathcal{P}_1$), then
\begin{eqnarray}
&& \frac{d}{dt}\|v(t)\|_H^2+C_1\|v(t)\|^2  \leq C_2\ \mbox{for a.e.}\ t\in \br^+,\label{eqn:ex2_estimate1}\\
&& \|v'(t)\|_{V^*}\leq C_3+C_4\|v(t)\|\ \mbox{for a.e.}\ t\in \br^+,\label{eqn:ex2_estimate2}
\end{eqnarray}
with $C_1,C_2,C_3,C_4>0$.
\end{lemma}
\begin{proof}
We take $z=v(t)$ in (\ref{def_sol_ex_2}) and obtain
$$
\frac{1}{2}\frac{d}{dt}\|v(t)\|_H^2+\|v(t)\|^2+(\xi(t),\gamma v(t))_{L^2(\Gamma_M)} \leq \langle F,v(t)\rangle,
$$
where $\xi(t)\in S^2_{\partial j(\gamma v(t))}$ for a.e. $t\in \br^+$.
From $H(j)(ii)$ we obtain with arbitrary $\varepsilon > 0$
$$
\frac{1}{2}\frac{d}{dt}\|v(t)\|_H^2+\|v(t)\|^2  \leq \frac{\varepsilon}{2}\|v(t)\|^2  + \frac{1}{2\varepsilon}\|F\|_{V^*}^2
+ d\|\gamma v(t)\|^2_{L^2(\Gamma_M)} - c m(\Gamma_m).
$$
 Taking $\varepsilon = 1 - d\|\gamma\|^2$ we obtain with a constant $C>0$
$$
\frac{1}{2}\frac{d}{dt}\|v(t)\|_H^2+\frac{1 - d\|\gamma\|^2}{2}\|v(t)\|^2  \leq C,
$$
which proves (\ref{eqn:ex2_estimate1}).

To show (\ref{eqn:ex2_estimate2}) let us observe that for $z\in V$ and a.e. $t\in \br^+$ we have with $\xi(t)\in S^2_{\partial j(\gamma v(t))}$
$$
\langle v'(t), z\rangle \leq \|A\|_{\C{L}(V;V^*)}\|v(t)\|\|z\| + \|F\|_{V^*}\|z\|+\|\xi(t)\|_{L^2(\Gamma_M)}\|\gamma\|\|z\|.
$$

Now (\ref{eqn:ex2_estimate2}) follows directly from the growth condition $H(j)(i)$ and trace inequality.
\end{proof}
\begin{lemma}\label{lemma:absorbing_1}
There exists $R_0>0$ such that for all bounded sets $B\subset H$ there exists a time $t_0>0$ such that if $v_0\in B$ then $\|v(t)\|_H\leq R_0$ for all $t\geq t_0$
and in consequence the ball $B(0,R_0)$ is absorbing in $H$.
\end{lemma}
\begin{proof}
From (\ref{eqn:ex2_estimate1}) we obtain
$$
\frac{1}{2}\frac{d}{dt}\|v(t)\|_H^2+C_1\lambda_1\|v(t)\|_H^2  \leq C_2,
$$
where $\lambda_1$ is the first eigenvalue of $A$ in $V$. From Gronwall Lemma we get for all $t\in \br^+$
\begin{equation}\label{eqn:pointwise_bound_ex_1}
\|v(t)\|_H^2\leq \|v_0\|_H^2e^{-C_1\lambda_1 t} + \frac{C_2}{C_1\lambda_1}.
\end{equation}
The assertion holds with $R_0=\frac{C_2}{C_1\lambda_1}+1$.
\end{proof}
\begin{lemma}\label{lemma:uniform}
If $v\in L^2_{loc}(\br^+;V)$ with $v'\in L^2_{loc}(\br^+;V^\star)$ solves Problem ($\mathcal{P}_1$), then for all  bounded sets $B\subset H$
there exists $t_0=t_0(B)$ such that if $v_0\in B$, then
\begin{eqnarray}
&&\sup_{t\geq t_0}\int_t^{t+1}\|v(t)\|^2\ dt\leq C_5,\label{eqn:ex2_estimate3}\\
&&\sup_{t\geq t_0}\int_t^{t+1}\|v'(t)\|_{V^*}^2\ dt\leq C_6.\label{eqn:ex2_estimate4}
\end{eqnarray}
with $C_5,C_6>0$.
\end{lemma}
\begin{proof}
Estimate (\ref{eqn:ex2_estimate3}) follows directly from the integration of (\ref{eqn:ex2_estimate1}) and application of (\ref{eqn:pointwise_bound_ex_1}).
Estimate (\ref{eqn:ex2_estimate4}) is a direct consequence of (\ref{eqn:ex2_estimate3}) and (\ref{eqn:ex2_estimate2}).
\end{proof}
Before we pass to the proof of the fact that the $m$-semiflow associated with the problem $(\C{P}_1)$ satisfies the flattening condition we recall a lemma which is a simple consequence of Proposition 1.10 from \cite{Rossi2003}.
\begin{lemma}\label{lem:uniform-integrability}
Let $H$ be a separable Banach space and $T>0$. If the set $\C{U}\subset L^2(0,T;H)$ is relatively compact then it is 2-uniformly integrable, i.e.
\begin{equation}\label{eqn:uniform}
\forall\ \lambda > 0\ \ \exists \ \delta > 0\ :\quad \forall J\subset(0,1)\ \ m(J)\leq \delta \Rightarrow \sup_{u\in \C{U}}\int_J\|u(s)\|_{H}^2\ ds \leq\lambda.
\end{equation}
\end{lemma}
\begin{lemma}\label{lemma:flattening_1}
For every bounded set $B\subset H$ and every $\varepsilon >0$ there exists $t_0$ and finite dimensional subspace $H_m\subset H$ such that,
denoting the projection onto $H_m$ by $P_m$, for every $w\in \bigcup_{t\geq t_0} G(t,B)$ we have
$$
\|(I-P_m)w\|_H\leq \varepsilon,
$$
and in consequence the semiflow $G$ satisfies the flattening condition.
\end{lemma}
\begin{proof}
Let $0<\lambda_1\leq\lambda_2\leq \ldots \leq \lambda_m\leq\ldots$ be the eigenvalues of $A$ and $\{u_m\}_{m=1}^\infty$ be the corresponding
sequence of eigenfunctions which are orthogonal in $V$ and orthonormal in $H$. We have $\lambda_m\to\infty$ as $m\to\infty$. By $P_m$ we denote the projector onto the space $H_m$ spanned by the first $m$ eigenfunctions. Since $w\in \bigcup_{t\geq t_0} G(t,B)$ then there exists a function $v$, a solution to $(\C{P}_1)$ such that $v(0)\in B$ and $v(t)=w$, where $t\geq t_0$ and $t_0$ will be determined later. We denote $v(t) = v_1(t)+v_2(t)$,
where $v_1(t)\in  H_m$ and $v_2(t)\in  H_m^\bot$. We take the duality in (\ref{eq:example_11}) with $v_2(t)$ and we obtain for a.e. $t\in \br^+$
$$
\frac{1}{2}\frac{d}{dt}\|v_2(t)\|_H^2 + \|v_2(t)\|^2 \leq \|F\|_{V^*}\|v_2(t)\| + \|\xi(t)\|_{L^2(\Gamma_M)}\|\gamma\|\|v_2(t)\|.
$$
Using Cauchy inequality with $\varepsilon$ gives, for any $\varepsilon>0$ and a.e. $t\in\br^+$
$$
\frac{1}{2}\frac{d}{dt}\|v_2(t)\|_H^2 + \|v_2(t)\|^2 \leq \frac{1}{2\varepsilon}\|F\|^2_{V^*}+\varepsilon\|v_2(t)\|^2 + \frac{1}{2\varepsilon}\|\xi(t)\|^2_{L^2(\Gamma_M)}\|\gamma\|^2.
$$
We take $\varepsilon=\frac{1}{2}$ and we have
$$
\frac{1}{2}\frac{d}{dt}\|v_2(t)\|_H^2 + \frac{1}{2}\|v_2(t)\|^2 \leq \|F\|^2_{V^*}+\|\xi(t)\|^2_{L^2(\Gamma_M)}\|\gamma\|^2,
$$
for a.e. $t\in \br^+$. Note that for $u\in H_m^\bot$ we have the Courant-Fischer formula $\|u\|^2\geq \lambda_{m+1}\|u\|_H^2$. From the growth condition $H(j)(i)$ we obtain
$$
\frac{d}{dt}\|v_2(t)\|_H^2 +\lambda_{m+1}\|v_2(t)\|_H^2 \leq D_1+D_2\|\gamma v(t)\|_{L^2(\Gamma_M)}^2,
$$
where $D_1,D_2>0$ are positive constants independent on $t$ and initial condition. From the Translated Gronwall Lemma \ref{lemma:gronwall_lukaszewicz} we get, for all $t\in \br^+$,
\begin{align}\label{eqn:relation}
&\|v_2(t+2)\|_H^2\leq \\
&\leq e^{-\lambda_{m+1}}\int_t^{t+1}\|v_2(s)\|_H^2\,ds+ \frac{D_1}{\lambda_{m+1}}+D_2e^{-\lambda_{m+1}(t+2)}\int_{t}^{t+2}e^{\lambda_{m+1}s}\|\gamma v(s)\|_{L^2(\Gamma_M)}^2\ ds.\nonumber
\end{align}
Now for any $\delta\in (0,2)$ we have
\begin{align*}
&\int_{t}^{t+2}e^{\lambda_{m+1}s}\|\gamma v(s)\|_{L^2(\Gamma_M)}^2\ ds =\\
& = \int_{t}^{t+2-\delta}e^{\lambda_{m+1}s}\|\gamma v(s)\|_{L^2(\Gamma_M)}^2\ ds + \int_{t+2-\delta}^{t+2}e^{\lambda_{m+1}s}\|\gamma v(s)\|_{L^2(\Gamma_M)}^2\ ds \leq\\
& \leq e^{\lambda_{m+1}(t+2-\delta)}\|\gamma\|^2 \int_{t}^{t+2}\|v(s)\|^2\ ds + e^{\lambda_{m+1}(t+2)}\int_{t+2-\delta}^{t+2}\|\gamma v(s)\|_{L^2(\Gamma_M)}^2\ ds.
\end{align*}
Substituting this last inequality into (\ref{eqn:relation}) we obtain from Lemmata \ref{lemma:absorbing_1} and \ref{lemma:uniform} for $t\geq t_0(B)$
\begin{equation}\label{eqn:flattening}
\|v_2(t+2)\|_H^2\leq e^{-\lambda_{m+1}}R_0^2+ \frac{D_1}{\lambda_{m+1}}+2D_2
e^{-\delta\lambda_{m+1}}\|\gamma\|^2 C_5 + D_2\int_{t+2-\delta}^{t+2}\|\gamma v(s)\|_{L^2(\Gamma_M)}^2\ ds,
\end{equation}
where $R_0$ is given by Lemma \ref{lemma:absorbing_1}. We denote $\C{W}(0,2)=\{u\in L^2(0,2;V)\  |\  u'\in L^2(0,2;V^*)\}$. By Lemma \ref{lemma:uniform} the set
$$
\C{U}=\{v(\cdot-t)|_{[0,2]}\ |\ v\ \mbox{solves}\  (\ref{def_sol_ex_2})\ \mbox{and}\ t\geq t_0 \}
$$
is bounded in $\C{W}(0,2)$ and hence (from Proposition 2.143 in \cite{carl}) the set $\gamma \C{U}$ is relatively compact in $L^2(0,2;L^2(\Gamma_M))$. From Lemma \ref{lem:uniform-integrability} it follows that this set is 2-uniformly integrable.
We take $\lambda=\frac{\varepsilon}{2D_2}$. From (\ref{eqn:uniform}) and (\ref{eqn:flattening}) it follows that there exists $\delta>0$ such that for all $t\geq t_0$
\begin{equation}\label{eqn:relation2}
\|v_2(t+2)\|_H^2\leq e^{-\lambda_{m+1}}R_0^2+ \frac{D_1}{\lambda_{m+1}}+2D_2
e^{-\delta\lambda_{m+1}}\|\gamma\|^2 C_5 + \frac{\varepsilon}{2}.
\end{equation}
Now it suffices to choose $m$ large enough such that the sum of first three terms on the right-hand side of (\ref{eqn:relation2}) is less than $\frac{\varepsilon}{2}$
and the proof is complete.
\end{proof}

\begin{lemma}\label{lemma:nw_1}
If $\{v_n\}_{n=1}^\infty$ solve Problem ($\mathcal{P}_1$) with the initial conditions $\{v_n^0\}$ such that $v_n^0\to v_0$ strongly in $H$ then
for any $t>0$ for a subsequence we have $v_n(t)\to v(t)$ weakly in $H$, where $v$ is a solution with the initial condition $v_0$, and, in consequence the semiflow $G$ satisfies the condition ($NW$).
\end{lemma}
\begin{proof}
The proof is standard since the a priori estimates of Lemma \ref{lemma:basic_a_priori} provide enough convergence to pass to the limit. Indeed, the sequence $v_n$ is bounded in $L^2(0,t;V)$ with $v_n'$ bounded in $L^2(0,t;V^*)$. Hence, for a subsequence $v_n\to v$ weakly in $L^2(0,t;V)$ with $v_n'\to v'$
weakly in $L^2(0,t;V^*)$. In consequence for all $s\in [0,t]$ we have $v_n(s)\to v(s)$ weakly in $H$, which means that $v(0)=v_0$ and $v_n(t)\to v(t)$ weakly in $H$. We must show
that $v$ solves Problem $(\C{P}_1)$ on $(0,t)$. We only discuss passing to the limit in multivalued term since for other terms it is standard. Let $\xi_n\in L^2(0,t;L^2(\Gamma_M))$ be such that $\xi_n(t)\in S^2_{\partial j(\gamma v_n(t))}$ and (\ref{eq:example_11}) holds. From the growth condition $H(j)(i)$ it follows that $\xi_n$ is bounded in $L^2(0,t;L^2(\Gamma_M))$ and thus, for a subsequence we have $\xi_n\to \xi$ weakly in $L^2(0,t;L^2(\Gamma_M))$. To conclude the proof we need to show that $\xi(s)\in S^2_{\partial j(\gamma v(s))}$ for a.e. $s\in (0,t)$. From the compactness of the Nemytskii trace operator (see \cite{carl}) it follows that $\gamma v_n\to \gamma v$ strongly in $L^2(0,t;L^2(\Gamma_M))$ and, in consequence, for a subsequence, $\gamma v_n(t)\to \gamma v(t)$ strongly in $L^2(\Gamma_M)$ for a.e. $t\in (0,T)$. Now let us observe that the multivalued mapping $L^2(\Gamma_M)\ni u\to S^2_{\partial j(u)}\in P(L^2(\Gamma_M))$ is strong - weak upper semicontinuous (see \cite{naniewicz-2004}, \cite{pk-gl}). The assertion follows from the convergence theorem of Aubin and Cellina (see Theorem 1.4.1 in \cite{aubin-cellina}) and the proof is complete.
\end{proof}

From Lemmata \ref{lemma:absorbing_1}, \ref{lemma:flattening_1} and \ref{lemma:nw_1} it follows that all assumptions of Corollary \ref{corollary:main} hold and hence we have shown the following Theorem

\begin{theorem}
The $m$-semiflow $G$ associated with Problem ($\mathcal{P}_1$) has a global attractor.
\end{theorem}

\subsection{Attractor in $L^2(\Omega)$ and $L^p(\Omega)$ for reaction-diffusion problem with multivalued semilinear source term.}\label{sec:l2p}
\noindent The example of this subsection is a generalization of semilinear reaction-diffusion problem considered for example in Section 11 of \cite{robinson-2001-infty} to the case where the source term is given by a multifunction.

\noindent Note that the existence of attractor in $L^2(\Omega)$ for the problem considered in this section can be shown using the method of strong-strong upper semicontinuity of $m$-semiflow and the monotonicity of auxiliary energy function after technical modification of the proofs in \cite{Kasyanov2011}, \cite{Kasyanov2012}. The proof of the attractor existence in $L^p(\Omega)$ is based on the argument from \cite{lukaszewicz-2010} where the $L^p$ pullback attractor existence for the process associated with nonautonomous reaction-diffusion equation is shown. 

\noindent Let $\Omega\subset \br^m$ be an open and bounded set with Lipschitz boundary and let $p>2$. The letter $q$ will denote the conjugate exponent given by $\frac{1}{p}+\frac{1}{q}=1$. We use the notation $V=H^1_0(\Omega)$ and $H=L^2(\Omega)$. The spaces $V\subset H\subset V^*$ constitute an evolution triple with continuous, dense and compact embeddings, by $\langle\cdot,\cdot\rangle$ we denote the duality in $V\times V^*$. We define an operator $A:V\to V^*$ as $\langle Au,v\rangle = \int_\Omega \nabla u(x)\cdot \nabla v(x)\ dx$, the same symbol $A$ is used to denote the Nemytskii operator leading from $L^2(S;V)$ to $L^2(S;V^*)$, where $S\subset \br^+$ is a bounded and open time interval. We use the norm in $V$ given as $\|v\|^2=\langle Av,v\rangle$. By $\lambda_1>0$ we denote the constant in Poincare inequality $\lambda_1\|v\|_H^2\leq \|v\|^2$ valid for $v\in V$. Let moreover $\C{V}(S)=L^2(S;V)\cap L^p(S;L^p(\Omega))$. This space, equipped with the norm $\|u\|_{\C{V}(S)}=\|u\|_{L^2(S;V)}+\|u\|_{L^p(S;L^p(\Omega))}$ is in duality with the space $\C{V}^*(S)=L^2(S;V^*)+L^q(S;L^q(\Omega))$, where the norm in the latter space given by
$$\|u\|_{\C{V}^*(S)} = \inf_{\substack{w\in L^2(S;V^*),\\ z\in L^q(S;L^q(\Omega)),\\ w+z=u}}\max\{\|w\|_{L^2(S;V^*)},\|z\|_{L^q(S;L^q(\Omega))}\}.$$
 Moreover, the space $\C{W}(S)=\{u\in \C{V}(S)\ | \ u'\in \C{V}^*(S)\}$ is embedded continuously in
$C(\bar{S};H)$ and compactly in $L^2(S;H)$ (see \cite{robinson-2001-infty}). We will use the notation $\C{W}_{loc}(\br^+)$ (respectively $\C{V}_{loc}(\br^+)$, $\C{V}_{loc}^*(\br^+)$) for the space of functions which belong to $\C{W}(0,T)$ (respectively $\C{V}(0,T)$, $\C{V}^*(0,T)$) for all $T>0$.

We formulate the following problem

\noindent \textbf{Problem ($\mathcal{P}_2$).} Find $v\in \C{W}_{loc}(\br^+)$ with $v(0)=v_0$
 such that
 \begin{eqnarray} \label{def_sol_ex_3}
    &&  v'+Av+\xi=F\ \ \mbox{in}\ \ \C{V}^*_{loc}(\br^+),\label{eq:example_33}\\
    &&  \xi(t)\in S^q_{\partial j(v(t))}\ \ \mbox{a.e.}\ \ t\in\br^+. \nonumber
 \end{eqnarray}

  The assumptions on problem data are the following
 \begin{itemize}
 \item[$H_0$:] $v_0\in H, F\in H$,
 \item[$H(j)$:] $j:\br\to \br$ is a locally Lipschitz function such that
 \begin{itemize}
 \item[$(i)$] $\partial j$ satisfies the growth condition $|\xi| \leq a+b|s|^{p-1}$ for all $s\in \br$ and $\xi\in \partial j(s)$ with $a,b>0$,
 \item[$(ii)$] $\partial j$ satisfies the dissipativity condition $\xi s \geq c + d |s|^p$ for all $s\in \br$ and $\xi\in \partial j(s)$ with $c\leq 0$ and $d>0$.
 \end{itemize}
 \end{itemize}

Problem ($\mathcal{P}_2$) is in fact the weak formulation of the following initial and boundary value problem
   \begin{eqnarray} \label{def_sol_classical_3}
    &&  v'(x,t)-\Delta v(x,t) + \xi(x,t) = F(x)\ \mbox{in}\ \Omega\times \br^+,\\
    &&  v(x,0) = v_0(x)\ \mbox{in}\ \Omega,\\
    &&  v(x,t) = 0 \ \mbox{on}\ \partial\Omega \times \br^+,\\
    &&  \xi(x,t) \in \partial j(v(x,t)) \ \mbox{on}\ \Omega\times \br^+.
 \end{eqnarray}

 Existence result for Problem ($\mathcal{P}_2$) uses the standard approximation technique (see \cite{Miettinen1999}, \cite{Migorski-Ochal2007}, \cite{Kalita2010}), however we briefly outline the proof since this is a new result of independent interest in the theory of hemivariational inequalities.

 \begin{theorem}\label{thm:ex3_existence}
 Problem ($\mathcal{P}_2$) has at least one solution.
 \end{theorem}
 \begin{proof}
 Let $\varrho\in C^\infty_0(\br)$ be a standard mollifier kernel such that $\mbox{supp}(\varrho)\subset (-1,1)$, $\int_{-1}^1\varrho(s)\ ds = 1$ and $\rho(s)\geq 0$ for all $s\in (-1,1)$. We define $\varrho_n(s)=n\varrho(ns)$ for $s\in \br$ and $n\in\mathbb{N}$. We consider $j_n:\br\to \br$ defined by a convolution
 $$
 j_n(r)=\int_\br \varrho_n(s)j(r-s)\ ds.
 $$
 Note that $j_n\in C^\infty(\br)$. Moreover by calculations analogous to the proofs of Lemmata 5 and 9 in \cite{Kalita2010} it follows that $j_n$ satisfy $H(j)$ (where the classical derivative $j_n'$ is taken in place of $\partial j$) with the constants $a,b,d>0$ and $c\in\br$ different then those for $j$ but independent on $n$. Next, we define a family of auxiliary problems parameterized by $n$

\noindent \textbf{Problem ($\mathcal{P}_2^n$).}
 \textit{Find $v_n\in \C{W}_{loc}(\br^+)$ with $v_n(0)=v_0$
 such that}
 \begin{equation}
    v_n'+Av_n+j_n'(v_n)=F\ \ \mbox{in}\ \ \C{V}^*_{loc}(\br^+).\label{def_sol_ex_aux_3}
 \end{equation}

 It is known (see \cite{robinson-2001-infty}, \cite{Zhong-Yang-Sun}, \cite{lukaszewicz-2010}) that above problems have, possibly nonunique, solutions: we denote the  sequence of corresponding solutions by $\{v_n\}_{n=1}^\infty$. Taking the duality in (\ref{def_sol_ex_aux_3}) with $v_n$ and using $H(j)(ii)$ we obtain, after calculations, that $\{v_n\}$ is bounded in $\C{V}(0,T)$ and in $L^\infty(0,T;H)$, where the bound can possibly tend to infinity as $T\to\infty$. Moreover a standard calculation that uses $H(j)(i)$ shows that $\{v_n'\}$ is bounded in $\C{V}^*(0,T)$. Hence, for a subsequence constructed by the diagonal argument, we have
 \begin{eqnarray}
 && v_n\to v\ \ \mbox{weakly in}\ \ \C{W}_{loc}(\br^+)\\
 && v_n\to v\ \ \mbox{weakly-* in}\ \ L^\infty_{loc}(\br^+).
 \end{eqnarray}
 From the fact that the Nemytskii operator $A:L^2(0,T;V)\to L^2(0,T;V^*)$ is weakly continuous for all $t\geq 0$ we obtain $Av_n\to Av$ weakly in $\C{V}^*_{loc}(\br^+)$. Application of the growth condition $H(j)(i)$ gives that
 $j_n'(v_n)$ is bounded in $L^q_{loc}(\br^+;L^q(\Omega))$ and hence, for a subsequence
 \begin{equation}\label{eqn:weak_moll}
 j_n'(v_n)\to \xi\ \ \mbox{weakly in}\ \ L^q_{loc}(\br^+;L^q(\Omega)).
 \end{equation}
 We can pass to the limit in (\ref{def_sol_ex_aux_3}) and obtain that
 $$
 v'+A(v)+\xi=F\ \ \mbox{in}\ \ \C{V}^*_{loc}(\br^+).
 $$
 In a standard way it follows that $v(0)=v_0$. It remains to show that $\xi(t)\in \partial j(v(t))$ for a.e. $t\in \br^+$. The proof of this fact follows the lines   of the proof of Step III of Theorem 1 in \cite{Miettinen1999}. Let us fix $T>0$ and denote $\Omega_T=(0,T)\times \Omega$. Since $\C{W}(0,T)$ embeds in $L^2(0,T;H)$ compactly it follows that, for a subsequence, $v_n(t,x)\to v(t,x)$ for a.e. $(t,x)\in \Omega_T$.
 By Egoroff's Theorem for each $\delta>0$ we can find the Lebesgue measurable set $N_\delta$ such that $m(N_\delta)<\delta$ and $v_n\to v$ uniformly on $\Omega_T\setminus N_\delta$. Hence for each $\varepsilon>0$ we can find $n_0\in \mathbb{N}$ such that for all $(t,x)\in \Omega_T\setminus N_\delta$ and $n\geq n_0$ we have 
 \begin{equation}\label{eq:bnd}
 |v_n(x,t)|\leq |v(x,t)| + \varepsilon \in L^p(\Omega_T\setminus N_\delta).
 \end{equation} 
  Morover from (\ref{eqn:weak_moll}) it follows that $j_n'(v_n)\to \xi$ weakly in $L^1(\Omega_T\setminus N_\delta)$. Now take $w\in L^\infty(\Omega_T\setminus N_\delta)$. We have
 $$
 \int_{\Omega_T\setminus N_\delta}\xi(t,x)w(t,x)\ dx\,dt=\lim_{n\to\infty}\int_{\Omega_T\setminus N_\delta}j_n'(v_n(t,x))w(t,x)\ dx\,dt.
 $$
 From $H(j)(i)$, which is valid for all $n$ with the same $a,b$ and from the bound (\ref{eq:bnd}) we can invoke the Fatou Lemma and obtain
 \begin{equation}
 \int\limits_{\Omega_T\setminus N_\delta}\xi(t,x)w(t,x)\ dx\,dt\leq \int_{\Omega_T\setminus N_\delta}\limsup_{\substack{n\to\infty\\ \lambda\to 0^+}}\frac{j_n(v_n(t,x)+\lambda w(t,x))-j_n(v_n(t,x))}{\lambda}\ dx\,dt.
 \end{equation}
 We can estimate the last integrand for a.e. $(t,x)\in \Omega_T\setminus N_\delta$ as follows
 \begin{eqnarray}
  &&\limsup_{\substack{n\to\infty\\ \lambda\to 0^+}}\frac{j_n(v_n(t,x)+\lambda w(t,x))-j_n(v_n(t,x))}{\lambda} =\nonumber\\ && = \limsup_{\substack{n\to\infty\\ \lambda\to 0^+}}\int\limits_{\br}\varrho_n(\tau)\frac{j(v_n(t,x)-\tau+\lambda w(t,x))-j(v_n(t,x)-\tau)}{\lambda}\ d\tau \leq\nonumber\\
&& \leq \limsup_{\substack{n\to\infty\\ \lambda\to 0^+\\ \tau\to 0}}\frac{j(v_n(t,x)-\tau+\lambda w(t,x))-j(v_n(t,x)-\tau)}{\lambda}= \nonumber\\
&& = \limsup_{\substack{n\to\infty\\ \lambda\to 0^+\\ \tau\to 0}}\frac{j(v(t,x)+v_n(t,x)-v(t,x)-\tau+\lambda w(t,x))-j(v(t,x)+v_n(t,x)-v(t,x)-\tau)}{\lambda}=\nonumber\\
&&=j^0(v(t,x);w(t,x)).\nonumber
 \end{eqnarray}
 Hence
 \begin{equation}
 \int\limits_{\Omega_T\setminus N_\delta} \xi(t,x)w(t,x)\ dx\, dt\leq \int\limits_{\Omega_T\setminus N_\delta}j^0(v(t,x);w(t,x))\ \,dx\,dt
 \end{equation}
 Since the choice of $w$ is arbitrary, from the definition of the generalized gradient we get
 $$
 \xi(t,x)\in \partial j(v(t,x))\ \ \mbox{a.e.}\ \ (t,x)\in\Omega_T\setminus N_\delta,
 $$
 but, since $\delta>0$ can be arbitrarily small and $m(N_\delta)<0$ the last inclusion must hold for a.e. $x\in \Omega_T$ and the proof is complete.
 \end{proof}

 Note, that by taking the test function $u=w\theta(t)$, where $\theta|_{[0,T]}\in C^1([0,T])$ with $\theta(t)=0$ for $t\geq T$ and $w\in V\cap L^p(\Omega)$, analogously
 to the argument in Section 23 of \cite{Zeidler_2A} or Section 8 of \cite{robinson-2001-infty}, it follows, that if $v\in \mathcal{W}_{loc}(\mathbb{R}^+)$ solves (\ref{def_sol_ex_3}), then for a.e. $t\in \mathbb{R^+}$ we have
 \begin{equation}\label{ex_3_a_e}
 v'(t)+Av(t)+\xi(t)=F,
 \end{equation}
 where the equality is understood in $V^*+L^q(\Omega)$.

 We prove some estimates that are satisfied by the solutions to Problem ($\mathcal{P}_2$).

 \begin{lemma}
 If $v\in \mathcal{W}(\mathbb{R}^+)$ solves Problem ($\mathcal{P}_2$), then, for a.e. $t\geq 0$ we have
 \begin{eqnarray}
 && \frac{d}{dt}\|v(t)\|_H^2+\|v(t)\|_V^2 + 2d\|v(t)\|_{L^p(\Omega)}^p\leq \frac{\|F\|^2_H}{\lambda_1} -2cm(\Omega), \label{first_inequality}\\
 &&  \frac{d}{dt}\|(|v(t)|-M)_+\|_{L^p(\Omega)}^p+\frac{dp}{8}\|(|v(t)|-M)_+\|_{L^{2p-2}(\Omega)}^{2p-2}+\nonumber\\
 &&+\frac{dp}{4}M^{p-2}\|(|v(t)|-M)_+\|_{L^p(\Omega)}^p\leq\frac{p}{d}\|F\|_H^2, \label{second_inequality}
 \end{eqnarray}
 where (\ref{second_inequality}) holds with any $M\geq \left(-\frac{2c}{d}\right)^{\frac{1}{p}}$.
 \end{lemma}
 \begin{proof}
 Proof of (\ref{first_inequality}) is straightforward. This inequality follows
 by taking the duality in (\ref{ex_3_a_e}) with $v(t)\in V\cap L^p(\Omega)$ for a.e. $t\in \br^+$.

 In order to show (\ref{second_inequality}) we start from a formal estimation
 and then we show a technique to prove it rigorously. We proceed analogously to the proof of Lemma 3.1 in \cite{lukaszewicz-2010}. Taking the duality in (\ref{ex_3_a_e}) with $(v(t)-M)_+^{p-1}$, where $(v(t)-M)_+$ is a positive part of $v(t)-M$, and $M>0$, we obtain for a.e. $t\in \br^+$
 \begin{eqnarray}
 &&\frac{1}{p}\frac{d}{dt}\|(v(t)-M)_+\|_{L^p(\Omega)}^p+(p-1)\int_{\Omega}|\nabla(v(t)-M)_+|^2
 |(v(t)-M)_+|^{p-2}\ dx+\nonumber\\
 &&+\int_{\Omega}\xi(t)(v(t)-M)_+^{p-1}\ dx = \int_{\Omega}F(v(t)-M)_+^{p-1}\ dx.\label{estimate_aux_plus}
 \end{eqnarray}
 We observe that from $H(j)(ii)$ it follows, that for $s\geq\left(-\frac{2c}{d}\right)^{\frac{1}{p}}$ we have $\xi \geq \frac{d}{2}s^{p-1}$ for all $\xi\in\partial j(s)$ and hence if we assume that $M\geq \left(-\frac{2c}{d}\right)^{\frac{1}{p}}$ we have for a.e. $(x,t)\in\Omega\times\br^+$
 \begin{align}
 &\xi(x,t)(v(x,t)-M)_+^{p-1}\geq \frac{d}{2}|v(x,t)|^{p-1}(v(x,t)-M)_+^{p-1} \geq\nonumber\\
 & \geq \frac{d}{4}M^{p-2}(v(x,t)-M)_+^p+\frac{d}{4}(v(x,t)-M)_+^{2p-2}.\label{eq_aux_1}
 \end{align}
 Moreover we have for a.e. $t\in \br^+$ and a.e. $x\in\Omega$
 \begin{equation}\label{eq_aux_2}
 F(x)(v(x,t)-M)_+^{p-1} \leq \frac{d}{8}(v(x,t)-M)_+^{2p-2}+\frac{2}{d}F(x)^2.
 \end{equation}
 Applying (\ref{eq_aux_1}) and (\ref{eq_aux_2}) to (\ref{estimate_aux_plus}) we obtain
  \begin{equation}\label{eq_aux_3}
 \frac{d}{dt}\|(v(t)-M)_+\|_{L^p(\Omega)}^p+\frac{dp}{8}\|(v(t)-M)_+\|_{L^{2p-2}(\Omega)}^{2p-2}+\frac{dp}{4}M^{p-2}\|(v(t)-M)_+\|_{L^p(\Omega)}^p\leq\frac{2p}{d}\|F\|_H^2.
 \end{equation}
 Taking $(v(t)+M)_-^{p-1}$ as the test function in (\ref{ex_3_a_e}), where $(v(t)+M)_-$ is the negative part of $v(t)+M$ exactly as above we get
  \begin{equation}\label{eq_aux_4}
 \frac{d}{dt}\|(v(t)+M)_-\|_{L^p(\Omega)}^p+\frac{dp}{8}\|(v(t)+M)_-\|_{L^{2p-2}(\Omega)}^{2p-2}+\frac{dp}{4}M^{p-2}\|(v(t)+M)_-\|_{L^p(\Omega)}^p\leq\frac{2p}{d}\|F\|_H^2.
 \end{equation}
 Combining (\ref{eq_aux_3}) and (\ref{eq_aux_4}) we obtain (\ref{second_inequality}).

 In order for the derivation of (\ref{second_inequality}) to be rigorous we need
 that $(v(t)+M)_-^{p-1}, (v(t)-M)_+^{p-1}\in V\cap L^p(\Omega)$. If this is not the case, we proceed analogously to the argument of Section III.6.2 in \cite{temam-infty}. First we define $\theta_N\in C^\infty([0,\infty))$ by $\theta_N(s)=s$ for $s\in [0,N]$, $\theta_N(s)=2N$ for $s\geq 2N$ and $\theta_N$ strictly increasing and smooth. We take $\theta_N((v(t)-M)_+^{p-1})$ as a test function in (\ref{ex_3_a_e})
 and obtain inequality similar to (\ref{eq_aux_3}). Then we let $N\to\infty$ in this
 inequality and obtain precisely (\ref{eq_aux_3}). Same argument can be used for (\ref{eq_aux_4}).
 \end{proof}

 Note that (\ref{second_inequality}) is similar to the inequality (3.4) in Lemma 3.1 in \cite{lukaszewicz-2010}. The difference is, that in (\ref{second_inequality}) there is an additional term with $L^{2p-2}$ norm. This term will be needed for $\omega$-limit compactness of the associated multivalued semigroup, since due to lack of regularity we cannot take $-\Delta v(t)$ as a test function and obtain estimate on $\|v(t)\|_V$ analogous to (3.3) in \cite{lukaszewicz-2010}.

 We can associate with Problem ($\mathcal{P}_2$) the multifunction
 $G:\br^+\times H\to P(H)$ that assigns to the time $t$ and initial condition $v_0$ the set of states attainable from $v_0$ after time $t$. Obviously, $G$ is an $m$-semiflow.

\begin{lemma}\label{eqn:ex3_bdd_H}
Let $t\geq 0$ and $v_0\in B\subset H$, where $\|w\|_H\leq R$ for $w\in B$. Let moreover $v\in G(t,v_0)$. Then
  \begin{equation}\label{eqn:ex3_bdd_H_ineq}
  \|v\|_H^2\leq e^{-\lambda_1t}R^2+\frac{\|F\|_H^2-2cm(\Omega)\lambda_1}{\lambda_1^2}.
  \end{equation}
  In consequence $G:\br^+\times H\to P(H)$ has a bounded absorbing set.
\end{lemma}
\begin{proof}
The proof is straightforward. Multiplying (\ref{first_inequality}) by the integrating factor $e^{\lambda_1 t}$, we get for a.e. $t\geq 0$
  $$
  \frac{d}{dt}(e^{\lambda_1 t}\|v(t)\|_H^2)\leq \left(\frac{\|F\|_H^2-2cm(\Omega)}{\lambda_1}\right)e^{\lambda_1t}.
  $$
  Hence, after integration from $0$ to $t$ we get (\ref{eqn:ex3_bdd_H_ineq}).
\end{proof}

 \begin{lemma}\label{lemma:ex3_absorbing_p}
 If $v_0\in L^p(\Omega)$ then $G(t,v_0)\subset L^p(\Omega)$ and hence $G$ is an $m$-semiflow
 on $L^p(\Omega)$. Moreover $G:\br^+\times L^p(\Omega)\to P(L^p(\Omega))$ has a bounded absorbing set in $L^p(\Omega)$.
 \end{lemma}
  \begin{proof}
  We omit the term with $L^{2p-2}$ norm in left-hand side of (\ref{second_inequality}) and multiply it by $e^{\frac{dp}{4}M^{p-2}t}$. We obtain for a.e. $s\in \br^+$
  $$
  \frac{d}{ds}(\|(|v(s)|-M)_+\|_{L^p(\Omega)}^p e^{\frac{dp}{4}M^{p-2}s})\leq \frac{p}{d}\|F\|_H^2e^{\frac{dp}{4}M^{p-2}s}.
  $$
  After integration from $0$ to $t$ this gives us
  \begin{eqnarray}
  &&  \|(|v(t)|-M)_+\|_{L^p(\Omega)}^p \leq e^{-\frac{dp}{4}M^{p-2}t}  \|(|v_0|-M)_+\|_{L^p(\Omega)}^p +  \frac{4}{M^{p-2}d^2}\|F\|_H^2\leq\nonumber\\
  && \leq e^{-\frac{dp}{4}M^{p-2}t}  \|v_0\|_{L^p(\Omega)}^p +  \frac{4}{M^{p-2}d^2}\|F\|_H^2.\label{eq:gronwall_1}
  \end{eqnarray}
  Now assume that $v\in G(t,v_0)$ for any $t\geq 0$ and $v_0\in L^p(\Omega)$. We have
\begin{align*}
 & \|v\|^p_{L^p(\Omega)}=\int_{\{|v|\leq M\}}|v(x)|^p\ dx+\int_{\{|v|> M\}}|v(x)|^p\, dx\leq\\
&\leq M^pm(\Omega)+\int_{\{|v|> M\}}(|v(x)|-M+M)^p\, dx\leq \\
&\leq M^pm(\Omega) + 2^{p-1}M^pm(\Omega)+2^{p-1}\int_{\Omega}(|v(x)|-M)_+^p\, dx.
 \end{align*}
  Using (\ref{eq:gronwall_1}) we obtain
  \begin{equation}\label{eq:bound_lp}
  \|v\|^p_{L^p(\Omega)}\leq M^pm(\Omega) (1+2^{p-1})+\frac{2^{p+1}}{M^{p-2}d^2}\|F\|_H^2+
  2^{p-1}e^{-\frac{dp}{4}M^{p-2}t}  \|v_0\|_{L^p(\Omega)}^p.
  \end{equation}
  The assertion follows directly from (\ref{eq:bound_lp}).
  \end{proof}

  \begin{lemma}\label{lemma:ex3_NW_H}
  The multivalued semiflow $G:\br^+\times H\to P(H)$ satisfies the
  condition ($NW$).
  \end{lemma}
  \begin{proof}
  Assume that $\{v_n\}_{n=1}^\infty$ solve Problem ($\mathcal{P}_2$) with the initial conditions $v_n^0$ where $v_n^0\to v_0$ strongly in $H$. Fix $t>0$. From (\ref{first_inequality}) it follows that $v_n$ is bounded in $\C{V}(0,t)$. Moreover,
  analogously to the argument in the proof of Theorem \ref{thm:ex3_existence} from
  growth condition $H(j)(i)$ and the bound on $v_n$ in $\C{V}(0,t)$ it follows that
  $v_n'$ is bounded in $\C{V}^*(0,t)$. Hence, by reflexivity of these spaces
  we can extract a subsequence, denoted by the same index $n$, such that $v_n\to v$ weakly in $\C{W}(0,t)$ for some $v\in \C{W}(0,t)$.
  Using the same argument as in the proof of Theorem \ref{thm:ex3_existence} we can pass
  to the limit in (\ref{eq:example_33}) written for $v_n$ and find that $v$ solves Problem ($\mathcal{P}_2$) with the initial condition $v_0$. Moreover, since $\C{W}(0,t)$ embeds continuously in $C([0,t];H)$, we have that $v_n\to v$ weakly in $C([0,t];H)$ and moreover $v_n(t)\to v(t)$ weakly in $H$.
  \end{proof}

  \begin{lemma}\label{lemma:ex3_NW_p}
  The multivalued semiflow $G:\br^+\times L^p(\Omega)\to P(L^p(\Omega))$ satisfies the
  condition ($NW$).
  \end{lemma}
  \begin{proof}
Let $t\geq 0$.  Assume that $v_{n}^0\to v_0$ strongly in $L^p(\Omega)$ and $v_n$ solve Problem ($\mathcal{P}_2$) with initial conditions $v_n^0$. Since, by Lemma \ref{lemma:ex3_NW_H}, $G$ satisfies the condition ($NW$) in $H$, we have, for a subsequence $v_n(t)\to v(t)$
  weakly in $H$, where $v$ solves ($\mathcal{P}_2$) with the initial condition $v_0$.
  Since, by (\ref{eq:bound_lp}), $v_n(t)$ is bounded in $L^p(\Omega)$, then, for another subsequence, it follows that $v_n(t)\to w$ weakly in $L^p(\Omega)$ for some
  $w\in L^p(\Omega)$. Hence it must be that $v_n(t)\to w$ weakly in $H$, and, from the uniqueness of the limit we have that $w=v(t)$, which completes the proof.
  \end{proof}

   The following lemma is an analogue of Lemma 3.3 in \cite{lukaszewicz-2010}.

  \begin{lemma}\label{lemma:epsilon}
  If $v$ solves Problem ($\mathcal{P}_2$) with $v_0\in B$, where $B\subset L^2(\Omega)$ is bounded, then for every $\varepsilon>0$
  there exists the large enough constant $M$ dependent on $\varepsilon$, $B$ and the problem data such that for all $t\geq 2$ we have
  $$
  \|(|v(t)|-M)_+\|_{L^p(\Omega)}^p\leq \varepsilon.
  $$
  \end{lemma}
  \begin{proof}
  From (\ref{second_inequality}) it follows by Lemma \ref{lemma:gronwall_lukaszewicz}
  that for all $t\geq 0$ we have
  \begin{equation}\label{eqn:gronwall_luk_applied}
  \|(|v(t+2)|-M)_+\|_{L^p(\Omega)}^p\leq e^{-\frac{dp}{4}M^{p-2}}\int_t^{t+1}\|(|v(s)|-M)_+\|_{L^p(\Omega)}^p\, ds + \frac{4\|F\|_H^2}{d^2M^{p-2}}.
  \end{equation}
  To estimate the integral in the right-hand side of (\ref{eqn:gronwall_luk_applied})
  observe that for a.e. $x\in \Omega$ and a.e. $s\in \br^+$ we have $(|v(x,s)|-M)_+^p\leq |v(x,s)|^p$ and hence
  \begin{equation}\label{eqn:ex3_estimate_trunc}
  \int_t^{t+1}\|(|v(s)|-M)_+\|_{L^p(\Omega)}^p\, ds\leq   \int_t^{t+1}\|v(s)\|_{L^p(\Omega)}^p\, ds.
  \end{equation}
  Integrating (\ref{first_inequality}) from $t$ to $t+1$ we obtain for all $t\geq 0$
  \begin{equation}\label{eqn:ex3_estimate_int}
  2d\int_t^{t+1}\|v(s)\|_{L^p(\Omega)}^p\, ds\leq \|v(t)\|_H^2-2cm(\Omega)+\frac{\|F\|_H^2}{\lambda_1}.
  \end{equation}
  Using Lemma \ref{eqn:ex3_bdd_H} we get
  $$
  2d\int_t^{t+1}\|v(s)\|_{L^p(\Omega)}^p\, ds\leq e^{-\lambda_1t}R^2+\frac{(\|F\|_H^2-2cm(\Omega)\lambda_1)(1+\lambda_1)}{\lambda_1^2}.
  $$
  Hence, and from (\ref{eqn:gronwall_luk_applied}) and (\ref{eqn:ex3_estimate_trunc})
  we obtain for all $t\geq 0$
  \begin{equation}\label{eqn:gronwall_luk_final}
  \|(|v(t+2)|-M)_+\|_{L^p(\Omega)}^p\leq e^{-\frac{dp}{4}M^{p-2}}\left(  \frac{R^2}{2d}+\frac{(\|F\|_H^2-2cm(\Omega)\lambda_1)(1+\lambda_1)}{2d\lambda_1^2} \right) + \frac{4\|F\|_H^2}{d^2M^{p-2}}.
  \end{equation}
  The proof is finished. Note that $M$ depends on $\varepsilon, d, c, p, R, m(\Omega), \|F\|_H, \lambda_1$.
  \end{proof}

  \begin{lemma}\label{lemma:ex3_flattening_H}
  The multivalued semiflow $G:\br\times H\to P(H)$ satisfies the flattening condition in $H$.
  \end{lemma}
  \begin{proof}
  The proof follows the lines of the proof of Lemma \ref{lemma:flattening_1}. We consider the sequence of eigenvalues of $A$ denoted by $0<\lambda_1\leq \lambda_2\leq \ldots$ and denote by $H_n$ the space spanned by first $n$ eigenfunctions and by $P_n:H\to H_n$ the projection operator. Then, if $v$ solves Problem ($\mathcal{P}_2$), we use the notation $v_1(t)=P_nv(t)$ and $v_2(t)=v(t)-v_1(t)$.
  Assume that $v_0\in B$, where $B$ is such that $\|w\|_H\leq R$ for $w\in B$ and $t\geq t_0\geq 2$.
  We take duality in (\ref{ex_3_a_e}) with $v_2(t)$ and obtain
  $$
  \frac{1}{2}\frac{d}{dt}\|v_2(t)\|_H^2+\|v_2(t)\|_V^2\leq \|F\|_H\|v_2(t)\|_H+\|\xi(t)\|_H\|v_2(t)\|_H.
  $$
  From Cauchy inequality and Poisson inequality we get
  $$
  \frac{d}{dt}\|v_2(t)\|_H^2+\|v_2(t)\|_V^2\leq \frac{1}{\lambda_1}(\|F\|^2_H+\|\xi(t)\|^2_H).
  $$
  Applying the Courant-Fischer formula we get
  $$
  \frac{d}{dt}\|v_2(t)\|_H^2+\lambda_{n+1}\|v_2(t)\|_H^2\leq \frac{1}{\lambda_1}(\|F\|^2_H+\|\xi(t)\|^2_H).
  $$
 From the Translated Gronwall Lemma \ref{lemma:gronwall_lukaszewicz} we get, for all $t\geq 0$,
  \begin{equation}\label{eqn:ex3_estimate_flattening}
  \|v_2(t+2)\|_H^2  \leq e^{-\lambda_{n+1}}\int_t^{t+1}\|v_2(s)\|_H^2\,ds + \frac{\|F\|_H^2}{\lambda_1\lambda_{n+1}} + \frac{e^{-\lambda_{n+1}(t+2)}}{\lambda_1}\int_t^{t+2}e^{\lambda_{n+1}s}\|\xi(s)\|^2_H\, ds.
  \end{equation}
  In order to estimate the last integral in right-hand side of above inequality we use the growth condition $H(j)(i)$. We have
  $$
  \|\xi(s)\|^2_H \leq \int_\Omega(a+b|v(x,s)|^{p-1})^2\, dx\leq
  2a^2m(\Omega) + 2b^2\int_\Omega |v(x,s)|^{2p-2}\, dx.
  $$
  Now observe that $|v(x,s)|^{2p-2}\leq 2^{2p-3}(M^{2p-2}+(|v(x,s)|-M)_+^{2p-2})$ and therefore we have
  $$
  \|\xi(s)\|^2_H \leq 2a^2m(\Omega) + b^2 2^{2p-2}\int_{\Omega}M^{2p-2}+(|v(x,s)|-M)_+^{2p-2}\, dx.
  $$
  Applying this last inequality in (\ref{eqn:ex3_estimate_flattening}) gives
  \begin{eqnarray}
&&  \|v_2(t+2)\|_H^2  \leq e^{-\lambda_{n+1}}\int_t^{t+1}\|v_2(s)\|_H^2\, ds  + \frac{\|F\|_H^2+m(\Omega)(2a^2 + b^2 2^{2p-2}M^{2p-2})}{\lambda_1\lambda_{n+1}}+\nonumber\\
&&+\frac{b^2 2^{2p-2}}{\lambda_1}e^{-\lambda_{n+1}(t+2)}\int_t^{t+2}e^{\lambda_{n+1}s}\|(|v(x,s)|-M)_+\|_{L^{2p-2}(\Omega)}^{2p-2}\, ds.\label{eqn:ex3_estimate_flattening2}
  \end{eqnarray}
  In order to estimate the last integral in (\ref{eqn:ex3_estimate_flattening2}) we multiply
  (\ref{second_inequality}) by $e^{\lambda_{n+1}t}$ and integrate from $t$ to $t+2$. We get
  \begin{align*}
  &\int_t^{t+2}e^{\lambda_{n+1}s}\frac{d}{ds}\|(|v(s)|-M)_+\|_{L^p(\Omega)}^p\, ds + \frac{dp}{8}\int_{t}^{t+2}e^{\lambda_{n+1}s}\|(|v(s)|-M)_+\|_{L^{2p-2}(\Omega)}^{2p-2}\, ds\leq \\
  &\leq \frac{p}{d}\|F\|_H^2\frac{e^{\lambda_{n+1}(t+2)}}{\lambda_{n+1}}.
  \end{align*}
  After integration by parts we get
  \begin{align*}
  & \frac{dp}{8}\int_{t}^{t+2}e^{\lambda_{n+1}s}\|(|v(s)|-M)_+\|_{L^{2p-2}(\Omega)}^{2p-2}\, ds\leq \frac{p}{d}\|F\|_H^2\frac{e^{\lambda_{n+1}(t+2)}}{\lambda_{n+1}}+\\
 & +e^{\lambda_{n+1}t}\|(|v(t)|-M)_+\|^p_{L^p(\Omega)}
   +\lambda_{n+1}\int_{t}^{t+2}e^{\lambda_{n+1}s}\|(|v(s)|-M)_+\|^p_{L^p(\Omega)}\, ds.
  \end{align*}
  By Lemma \ref{lemma:epsilon}, for any small $\delta<0$ we are able to choose large enough $M=M(\delta,R)$ such that for all $t\geq t_0$ we have $\|(|v(t)|-M)_+\|^p_{L^p(\Omega)}\leq \delta$. We have after straightforward calculation
  $$
   \frac{dp}{8}\int_{t}^{t+2}e^{\lambda_{n+1}s}\|(|v(s)|-M)_+\|_{L^{2p-2}(\Omega)}^{2p-2}\, ds\leq \frac{p}{d}\|F\|_H^2\frac{e^{\lambda_{n+1}(t+2)}}{\lambda_{n+1}}+2e^{\lambda_{n+1}(t+2)}\delta.
   $$
   Using this inequality in (\ref{eqn:ex3_estimate_flattening2}) we find
     \begin{eqnarray}
&&  \|v_2(t+2)\|_H^2  \leq e^{-\lambda_{n+1}}\int_t^{t+1}\|v_2(s)\|_H^2\, ds + \frac{\|F\|_H^2+m(\Omega)(2a^2 + b^2 2^{2p-2}M^{2p-2})}{\lambda_1\lambda_{n+1}}+\nonumber\\
&&+ \frac{b^2 2^{2p-2}}{\lambda_1}\left( \frac{8\|F\|^2_H}{d^2 \lambda_{n+1}}+\frac{16 \delta}{dp}\right).\label{eqn:ex3-bound-asymptotic}
  \end{eqnarray}
By Lemma \ref{eqn:ex3_bdd_H} we obtain for all $t\geq 0$
$$
\int_t^{t+1}\|v_2(s)\|_H^2\, ds\leq \int_t^{t+1} \|v(s)\|_H^2\, ds \leq e^{-\lambda_1t}R^2+\frac{\|F\|_H^2-2cm(\Omega)\lambda_1}{\lambda_1^2}.
$$
Using this inequality in (\ref{eqn:ex3-bound-asymptotic}) gives
     \begin{eqnarray}
&&  \|v_2(t+2)\|_H^2  \leq e^{-\lambda_{n+1}}e^{-\lambda_1t}R^2+ e^{-\lambda_{n+1}} \frac{\|F\|_H^2-2cm(\Omega)\lambda_1}{\lambda_1^2} + \frac{\|F\|_H^2+2a^2m(\Omega)}{\lambda_1\lambda_{n+1}} + \nonumber\\
&&+ \frac{b^2 2^{2p-2}M^{2p-2}m(\Omega)}{\lambda_1\lambda_{n+1}}+ \frac{b^2 2^{2p+1}\|F\|^2_H}{d^2\lambda_1\lambda_{n+1}} + \frac{b^2 2^{2p+2}\delta}{dp\lambda_1}:=I_1+I_2+I_3+I_4+I_5+I_6.\nonumber
  \end{eqnarray}
  We choose $\delta$ such that $I_6\leq \frac{\varepsilon}{6}$. This fixes (possibly large) $M$. Next we choose $n$ such that $\lambda_{n+1}$ is large enough and hence $I_i\leq \frac{\varepsilon}{6}$ for $i=2,3,4,5$. Finally we choose $t_0$ large enough so that $I_1\leq \frac{\varepsilon}{6}$ for all $t\geq t_0$ and the proof is complete.
  \end{proof}
  \begin{theorem}
  The multivalued semiflow $G:\br^+\times H\to P(H)$ has a global attractor in $H$.
  \end{theorem}
  \begin{proof}
  The theorem follows from Lemmata \ref{eqn:ex3_bdd_H}, \ref{lemma:ex3_NW_H}, \ref{lemma:ex3_flattening_H} and Corollary \ref{corollary:main}.
  \end{proof}

  We formulate the following lemma which is a corollary of Lemma 5.3 in \cite{Zhong-Yang-Sun} (compare Corollary 5.4 in \cite{Zhong-Yang-Sun} and Lemma 2.1 in \cite{lukaszewicz-2010}).
  \begin{lemma}\label{lemma:ex3_omega_compact}
  Let $G$ be a multivalued semiflow in $H$ and $L^p(\Omega)$ ($p\geq 2$) such that
  \begin{itemize}
  \item[$(i)$] $G:\br^+\times H\to P(H)$ is $\omega$-limit compact in $H$,
  \item[$(ii)$] for any $\overline{\varepsilon}>0$ and any bounded subset $B\subset L^p(\Omega)$ there exists a positive constant $\bar{M}=\bar{M}(\overline{\varepsilon},B)$ and $T=T(\overline{\varepsilon},B)$ such that
      for all $v_0\in B$, $t\geq T$ and for all $v\in G(t,v_0)$ we have
      $$
      \int_{\{|v|\geq \overline{M}\}} |v(x)|^p \, dx \leq \overline{\varepsilon}.
      $$
  \end{itemize}
  Then $G:\br^+\times L^p(\Omega)\to P(L^p(\Omega))$ is $\omega$-limit compact in $L^p(\Omega)$.
  \end{lemma}

  \begin{theorem}
  The multivalued semiflow $G:\br^+\times L^p(\Omega)\to P(L^p(\Omega))$ has a global attractor in $L^p(\Omega)$.
  \end{theorem}
  \begin{proof}
  From Lemmata \ref{lemma:ex3_absorbing_p} and \ref{lemma:ex3_NW_p} it follows $G$ satisfies condition ($NW$) and has a bounded absorbing set in $L^p(\Omega)$. From Lemma \ref{lemma:ex3_flattening_H} it follows that $G$ satisfies the flattening condition in $H$ and, by Lemma \ref{thm-flat-limit}, we deduce that $G$ is $\omega$-limit compact in $H$. To obtain $\omega$-limit compactness in $L^p(\Omega)$ we will use Lemma \ref{lemma:ex3_omega_compact}. To this end let $t\geq 2$ and $v_0\in B$, where $B$ is bounded in $L^p(\Omega)$.  We take $\varepsilon$ and $M$ as in Lemma \ref{lemma:epsilon}.
  We have for $v\in G(t,v_0)$
  \begin{align*}
  &\int_{\{|v|\geq 2M\}}|v(x)|^p\, dx\leq   \int_{\{|v|\geq 2M\}}(|v(x)|-M+M)^p\, dx\leq\\
 & \leq 2^{p-1}\left(\int_{\{|v|\geq 2M\}}(|v(x)|-M)^p\, dx+\int_{\{|v|\geq 2M\}}M^p\, dx\right)\leq\\
  &\leq 2^{p-1}2\int_{\{|v|\geq M\}}(|v(x)|-M)_+^p\, dx\leq 2^p\varepsilon,
  \end{align*}
  as $M\leq |v|-M$ for $|v|\geq 2M$. This proves that $(ii)$ of Lemma \ref{lemma:ex3_omega_compact} holds. Hence $G$ is $\omega$-limit compact in $L^p(\Omega)$ and from
  Theorem \ref{thm-main1} it follows that $G$ has an attractor in $L^p(\Omega)$.
  \end{proof}

\section*{Acknowledgments} The authors would like to thank Nicolaos S. Papageorgiou for his several remarks during the preparation of this paper and also the unknown referee whose relevant comments allowed to improve the paper in some important points.


\addcontentsline{toc}{chapter}{Bibliographie}
\begin{thebibliography}{00}

\bibitem{Arrieta-2006} {\sc J.M. Arrieta, A. Rodriguez-Bernal, J. Valero,} {\em Dynamics of a reaction-diffusion equation with a discontinuous
nonlinearity}, Int. J. Bifurcation and Chaos, 16 (2006), 2695--2984.

        \bibitem{aubin-cellina} {\sc J.-P. Aubin, A. Cellina},
        {\em Differential Inclusions},
        Springer, Berlin, 1984.


\bibitem{babin1985} {\sc A.V. Babin and M.I. Vishik,} {\em Maximal attractors of semigroups corresponding to evolution differential equations,} Mat. Sb. (N.S.), 126 (1985), 397--419.

\bibitem{Balibrea-2012} {\sc F. Balibrea, T. Caraballo, P.E. Kloeden, J. Valero,} {\em Recent developments in dynamical systems: three perspectives,}
International Journal of Bifurcation and Chaos, 20 (2010), 2591--2636.

\bibitem{Ball-1997} {\sc J.M. Ball}, {\em Continuity properties and global attractors of generalized semiflows and the Navier-Stokes equations},
 Nonlinear Science, 7 (1997), 475--502.

\bibitem{gl-mb1} {\sc M. Boukrouche, G. {\L}ukaszewicz},
        {\em  On the existence of pullback attractor for a two-dimensional shear flow with Tresca’s
         boundary condition},
         Banach Center Publications, 81 (2008), 81--93.


\bibitem{Caraballo-2003} {\sc T. Caraballo, P. Martin-Rubio, J.C. Robinson},
{\em A comparison between two theories for multivalued semiflows and their asumptotic behavior,}
Set-Valued Analysis, 11 (2003), 297--322.


\bibitem{carl} {\sc S. Carl, V.K. Le, D. Motreanu},
        {\em Nonsmooth Variational Problems and Their Inequalities},
        Springer, New York, 2007.


\bibitem{Chepyzhov1995} {\sc V.V. Chepyzhov, M.I. Vishik,} {\em Trajectory attractors for evolution equations,} CR Acad Sci Paris Ser. I Math., 321 (1995), 1309--1314.


\bibitem{Chepyzhov1997} {\sc V.V. Chepyzhov,  M.I. Vishik,} {\em Evolution equations and their trajectory attractors}, J. Math. Pures Appl., 76 (1997), 913--964.


\bibitem{chep-vish-2002} {\sc V.V. Chepyzhov, M.I. Vishik},
        {\em Attractors for Equations of Mathematical Physics},
        AMS, Providence, RI, 2002.



\bibitem{cholewa-dlotko 2000}  {\sc J. Cholewa, T. D{\l}otko},
        {\em Global Attractors in Abstract Parabolic Problems},
        Cambridge University Press, Cambridge, UK, 2000.


\bibitem{Clarke} {\sc F.H. Clarke},
         {\em Optimization and nonsmooth analysis},
        SIAM, Philadelphia, 1990.

\bibitem{coti-zelati} {\sc M. Coti Zelati},
        {\em On the theory of global attractors and Lyapunov functionals},
Set-Valued and Variational Analysis,
21 (2013), 127--149.

\bibitem{deimling-1985} {\sc K.Deimling},
        {\em Nonlinear Funcional Analysis},
        Springer, Berlin, 1985.

\bibitem{denkowski2003} {\sc Z. Denkowski, S. Migorski, N.S. Papageorgiou}, {\em An Introduction to Nonlinear Analysis: Theory},
Kluwer Academic Publishers, Boston, 2003.

\bibitem{hale-1988} {\sc J.K. Hale},
        {\em Asymptotic Behavior of Dissipative Systems},
        AMS, Providence, RI, 1988.

\bibitem{Kalita2010} {\sc P. Kalita},
         {\em Decay of energy for second-order boundary hemivariational inequalities with coercive damping},
        Nonlinear Analysis TMA, 74 (2011), 1164--1181.

\bibitem{Kalita2012submitted} {\sc P. Kalita}, {\em Semidiscrete $\theta$-scheme for a parabolic operator
differential inclusion,} submitted.

\bibitem{pk-gl} {\sc P.Kalita, G.{\L}ukaszewicz},
        {\em Attractors for some Navier-Stokes flows with nonmonotone boundary conditions
governed by a hemivariational inequality,} arXiv: http://arxiv.org/abs/1307.3496.

\bibitem{Kapustyan2010} {\sc O.V. Kapustyan, J. Valero,} {\em Comparison between trajectory and global attractors for evolution systems without uniqueness of solutions,} International Journal of Bifurcation and Chaos, 20 (2010), 2723--2734.

\bibitem{Kasyanov2011} {\sc P.O. Kasyanov}, {\em Multivalued dynamics of solutions of an autonomous differential-operator inclusion with pseudomonotone nonlinearity},  Cybernetics and Systems Analysis, 47 (2011), 800--811.

\bibitem{Kasyanov2012} {\sc P.O. Kasyanov}, {\em Multivalued dynamics of solutions of autonomous operator differential equations with pseudomonotone nonlinearity,}  Mathematical Notes, 92 (2012), 205--218.

\bibitem{Kasyanov2012b} {\sc P.O. Kasyanov, L. Toscano, N.V. Zadoianchuk}, {\em Long-time behaviour of solutions for autonomous
evolution hemivariational inequality with
multidimensional ''reaction-displacement'' law,}  Abstract and Applied Analysis, 2012 (2012), Article ID 450984.

 \bibitem{Kasyanov2013} {\sc P.O. Kasyanov, L. Toscano, N.V. Zadoianchuk} {\em  Regularity of weak solutions and their attractors for a parabolic feedback control problem,} Set-Valued and Variational Analysis, 2013, DOI: 10.1007/s11228-013-0233-8.


\bibitem{Kapustyan2012}{\sc O.V. Kapustyan, P.O. Kasyanov, J. Valero,} {\em Structure and regularity of the global attractor of a reaction-diffusion equation with non-smooth nonlinear term}, arXiv: http://arxiv.org/abs/1209.2010.

\bibitem{lukaszewicz-2010} {\sc G. \L{}ukaszewicz}, {\em On pullback attractors in $L^p$ for nonautonomous
reaction-diffusion equations,} Nonlinear Analysis: Theory, Methods \& Applications, 73 (2010), 350--357.

\bibitem{Ma-Wang-Zhong} {\sc Q.F. Ma, S.H. Wang, C.K. Zhong},
        {\em Necessary and sufficient conditions for the existence of global attractors for semigroups and
        applications},
        Indiana University Math. J., 51 (2002), 1541--1559.

\bibitem{malek-necas-1996} {\sc J. Malek, J. Ne\v{c}as},
        {\em A finite-dimensional attractor for three-dimensional flow of incompressible fluids},
        J. Diff. Eqns., 127 (1996), 498--518.

\bibitem{Melnik-1998} {\sc V.S. Melnik, J. Valero}, {\em On attractors of multivalued semiflows and differential inclusions,} Set-Valued Analysis,
6 (1998), 83--111.

\bibitem{Melnik-2008} {\sc V.S. Melnik, J. Valero}, {\em Addendum to ''On Attractors of Multivalued
Semiflows and Differential Inclusions'' [Set-Valued Anal., 6 (1998), 83-111],} Set-Valued Analysis,
16 (2008), 507--509.

\bibitem{Miettinen1999} {\sc M. Miettinen, P.D. Panagiotopoulos},{\em On parabolic hemivariational inequalities and applications}, Nonlinear Anal., 35 (1999), 885--915.


\bibitem{Migorski2013} {\sc S. Migorski, A. Ochal, M. Sofonea},{\em Nonlinear Inclusions and Hemivariational Inequalities. Models and Analysis of Contact Problems}, Advances in Mechanics and Mathematics, vol. 26, Springer, New York, 2013, pages: 285, ISBN: 978-1-4614-4231-8.

\bibitem{Migorski2006} {\sc S. Migorski, A. Ochal}, {\em A Unified Approach to Dynamic Contact Problems in Viscoelasticity,} Journal of Elasticity, 83 (2006), 247-276.

\bibitem{Migorski2004} {\sc S. Migorski, A. Ochal}, {\em Boundary hemivariational inequality
of parabolic type,} Nonlinear Analysis, 57 (2004), 579--596.

\bibitem{Migorski-Ochal2007} {\sc S. Mig\'orski, A. Ochal},
        {\em Navier-Stokes models modeled by evolution hemivariational inequalities},
        Discrete and Continuous Dynamical Systems Supplement (2007), 731--740.

\bibitem{naniewicz-2004} {\sc Z. Naniewicz},
{\em Hemivariational inequalities governed by the p-Laplacian - Dirichlet
problem}, Control and Cybernetics, 33, 181-210, (2004).


\bibitem{Naniewicz1995} {\sc Z. Naniewicz, P.D. Panagiotopoulos}, {\em Mathematical Theory of Hemivariational Inequalities and Applications}, Dekker, New York, 1995.

\bibitem{robinson-2011-dim} {\sc J.C. Robinson},
        {\em Dimensions, Embeddings, and Attractors},
        Cambridge University Press, UK, 2011.

\bibitem{robinson-2001-infty}  {\sc  J.C. Robinson},
        {\em Infnite-dimensional Dynamical Systems},
        Cambridge University Press, Cambridge, UK 2001.


\bibitem{Rossi2003} {\sc R. Rossi, G. Savare}, {\em Tightness, integral equicontinuity and compactness for evolution problems in Banach spaces,}
Ann. Sc. Norm. Sup., Pisa, 2 (2003) 395-431.

\bibitem{Sell}
        {\sc G.R. Sell},
        {\em Global attractors for the three dimensional Navier-Stokes equations},
        J. Dyn. Diff. Eqs, 8 (1996), 1--33.

\bibitem{millor-sofonea-telega-2010} {\sc M. Shillor, M.Sofonea, J.J. Telega},
        {\em Models and Analysis of Quasistatic Contact: Variational Methods},
        Springer-Verlag, Berlin Heidelberg, 2010.


\bibitem{temam-infty} {\sc R. Temam}, {\em Infinite-Dimensional Dynamical Systems in Mechanics and Physics}, 2nd. ed.,
        Springer-Verlag, New York, 1997.

\bibitem{Wang-Zhou-2007} {\sc Yejuan Wang, Shengfan Zhou},
        {\em Kernel sections and uniform attractors of multivalued semiprocesses},
        J. Differential Equations, 232 (2007), 573--622.

\bibitem{wang-yang-2010} {\sc G. Wang, X. Yang,} {\em
Finite difference approximation of a parabolic hemivariational inequalities arising from temperature control problem,}
Int. J. Numer. Anal. Mod., 7 (2010), 108--124.

\bibitem{Zeidler_2A} {\sc E. Zeidler}, {\em Nonlinear Functional Analysis and its Applications Vol II/A: Linear Monotone Operators}, Springer, New York, 1992.

\bibitem{Zgurovsky2012a} {\sc M.Z. Zgurovsky, P.O. Kasyanov, N.V. Zadoianchuk (Zadoyanchuk)}, {\em Long-time behavior of solutions for quasilinear hyperbolic hemivariational inequalities with application to piezoelectricity problem}, Applied Mathematics Letters, 25 (2012), 1569--1574.

\bibitem{Zgurovsky2012} {\sc M.Z. Zgurovsky, P.O. Kasyanov, O.V. Kapustyan, J. Valero, N.V. Zadoianchuk}, {\em Evolution inclusions and Variation
Inequalities for Earth Data Processing III}, Springer-Verlag, Berlin, Heidelberg, 2012.

\bibitem{Zhong-Yang-Sun} {\sc Cheng-Kui Zhong, Mei-Hua Yang, Chun-You Sun},
        {\em The existence of global attractors for the norm-to-weak continuous semigroup and application to
        the nonlinear reaction-diffusion equations},
        J. Differential Equations, 223 (2006) 367--399.



\end{thebibliography}
\end{document}